\documentclass[a4paper,10pt]{amsart}

\usepackage{a4wide}
\usepackage[latin1]{inputenc}
\usepackage{amssymb}
\usepackage{color}
\usepackage{tikz}
\usepackage{enumerate}

\usepackage{url}
\usepackage[bookmarks=false,pdfborder={0 0 0.05}]{hyperref}
\hypersetup{colorlinks=true, citecolor=darkblue, linkcolor=darkblue}

\definecolor{darkblue}{rgb}{0,0,0.7} 
\newcommand{\darkblue}{\color{darkblue}} 
\newcommand{\Dfn}[1]{\emph{\darkblue #1}} 

\newtheorem{theorem}{Theorem}[section]
\newtheorem{proposition}[theorem]{Proposition}
\newtheorem{corollary}[theorem]{Corollary}
\newtheorem{lemma}[theorem]{Lemma}
\newtheorem{algorithm}[theorem]{Algorithm}

\theoremstyle{definition}
\newtheorem{definition}[theorem]{Definition}
\newtheorem{example}[theorem]{Example}
\newtheorem{conjecture}[theorem]{Conjecture}
\newtheorem{openproblem}[theorem]{Open Problem}
\newtheorem{remark}[theorem]{Remark}

\newcommand\woword{\mathbf w _\circ}
\newcommand\wo{w_\circ}
\newcommand\woc{{\woword(\c)}}
\newcommand\cwoc{{\c\woc}}
\newcommand\ckwoc{{\c^k\woc}}
\newcommand\Id{{\bf 1}}

\def\symdiag#1{[#1]_{\operatorname{sym}}}

\def\Dm{\Delta_m}
\def\Dmk{\Delta_{m,k}}
\def\DWk{\Delta^{k}_c(W)}
\def\Dk#1{\Delta^{k}_c(#1)}
\def\Phipm{\Phi_{\geq -1}}
\def\S#1{\mathcal{S}_{#1}}
\def\inv{\operatorname{inv}}
\def\c{{\bf c}}
\def\Lrc{\mathsf{Lr}_c}
\def\r{{\mathsf r}}
\def\Lrvar#1{{\mathsf{Lr}_{#1}}}
\newcommand{\rotatedword}[2]{#1_{\stackrel{\circlearrowleft}{#2}}}

\def\Z{\mathbb{Z}}
\title{Subword complexes, cluster complexes, and generalized multi-associahedra}

\author[C.~Ceballos]{Cesar Ceballos$^{ \dagger}$}
\address[C.~Ceballos]{Inst.\ Mathematics, FU Berlin, Arnimallee 2, 14195 Berlin, Germany.}
\email{ceballos@math.fu-berlin.de}
\thanks{$^\dagger$supported by DFG via the Research Training Group ``Methods for Discrete Structures" and the Berlin Mathematical School}

\author[J.-P.~Labb\'e]{Jean-Philippe Labb\'e$^{ \ddagger}$}
\address[J.-P.~Labb\'e]{Inst.\ Mathematics, FU Berlin, Arnimallee 2, 14195 Berlin, Germany.}
\email{labbe@math.fu-berlin.de}
\thanks{$^\ddagger$supported by a FQRNT Doctoral scholarship.}

\author[C.~Stump]{Christian Stump$^{ *}$}
\address[C.~Stump]{Institut f\"ur Algebra, Zahlentheorie, Diskrete Mathematik, Leibniz Universit\"at Hannover, Germany}
\email{stump@math.uni-hannover.de}
\thanks{$^*$partially supported by a CRM-ISM postdoctoral fellowship and the DFG via the Research Group \lq\lq Methods for Discrete Structures\rq\rq.}


\begin{document}

\begin{abstract}
In this paper, we use subword complexes to provide a uniform approach to finite type cluster complexes and multi-associahedra. We introduce, for any finite Coxeter group and any nonnegative integer $k$, a spherical subword complex called \emph{multi-cluster complex}. For~$k=1$, we show that this subword complex is isomorphic to the cluster complex of the given type. We show that multi-cluster complexes of types $A$ and $B$ coincide with known simplicial complexes, namely with the simplicial complexes of multi-triangulations and centrally symmetric multi-triangulations respectively. Furthermore, we show that the multi-cluster complex is \emph{universal} in the sense that every spherical subword complex can be realized as a link of a face of the multi-cluster complex. 
\end{abstract}

\maketitle

%
%

\section{Introduction}

\emph{Cluster complexes} were introduced by S. Fomin and A. Zelevinsky to encode exchange graphs of cluster algebras~\cite{fomin_y-systems_2003}. N.~Reading then showed that the definition of cluster complexes can be extended to all finite Coxeter groups~\cite{reading_clusters_2007,reading_sortable_2007}. In this article, we present a new combinatorial description of cluster complexes using \emph{subword complexes}. These were introduced by A. Knutson and E. Miller, first in type~$A$ to study the combinatorics of determinantal ideals and Schubert polynomials \cite{knutson_grobner_2005}, and then for all Coxeter groups in \cite{knutson_subword_2004}. We provide, for any finite Coxeter group $W$ and any Coxeter element $c \in W$, a subword complex which is isomorphic to the $c$-cluster complex of the corresponding type, and we thus obtain an explicit type-free characterization of $c$-clusters. This characterization generalizes a description for crystallographic types obtained by K.~Igusa and R. Schiffler in the context of cluster categories \cite{igusa_exceptional_2010}. The present approach allows us to define a new family of simplicial complexes by introducing an additional parameter $k$, such that one obtains $c$-cluster complexes for $k=1$. In type~$A$, this simplicial complex turns out to be isomorphic to the simplicial complex of multi-triangulations of a convex polygon which was described by C.~Stump in \cite{stump_new_2011} (see also~\cite{SS2010}), and, in a similar manner, by V.~Pilaud and M.~Pocchiola in the framework of sorting networks~\cite{pilaud_multitriangulations_2010}. In type~$B$, we obtain that this simplicial complex is isomorphic to the simplicial complex of centrally symmetric multi-triangulations of a regular convex polygon. Therefore, we call them \emph{multi-cluster complexes}. 
Besides, there is a \lq\lq naive\rq\rq\ way of generalizing the
cluster complex by considering the simplicial complex of sets of
almost positive roots which do not contain any $k+1$ pairwise not
compatible roots. In fact, this alternative definition lacks basic
properties of the cluster complex, see Remark~\ref{rem:naive_multi_def}.
The multi-cluster complexes introduced in this paper
are different from \emph{generalized cluster complexes} as defined by S.~Fomin and N.~Reading~\cite{fomin_generalized_2005}.
In the generalized cluster complex, the vertices are given by the simple negative roots together with several distinguished copies of the positive roots, while the vertices of the multi-cluster complex correspond to the positive roots together with several distinguished copies of the simple negative roots. Multi-cluster complexes turn out to be intimately related to Auslander-Reiten quivers and repetition quivers~\cite{gabriel_representations_1997}. In particular, the Auslander-Reiten translate on facets of multi-cluster complexes in types $A$ and $B$ corresponds to cyclic rotation of (centrally symmetric) multi-triangulations. Furthermore, multi-cluster complexes uniformize questions about multi-triangulations, subword complexes, and cluster complexes. One important example concerns the open problem of realizing the simplicial complexes of (centrally symmetric) multi-triangulations and spherical subword complexes as boundary complexes of convex polytopes.

\vspace*{10pt}

In Section~\ref{sec:mainresults}, we recall the various objects in question, namely multi-triangulations, subword complexes, and cluster complexes. Moreover, the main results are presented and the multi-cluster complex is defined (Definition~\ref{def:multi-cluster complex}). In Section~\ref{sec:spherical subword complexes}, we study flips on spherical subword complexes and present two natural isomorphisms between subword complexes whose word differ by commutation or by rotation of letters. In Section~\ref{sec:sorting}, we prove that the multi-cluster complex is independent of the choice of the Coxeter element (Theorem~\ref{theorem:independent_of_c}). Section~\ref{sec:maintheorem} contains a proof that the multi-cluster complex is isomorphic to the cluster complex for $k=1$ (Theorem~\ref{th:main theorem}). In Section~\ref{sec:polytopality}, we discuss possible generalizations of associahedra using subword complexes; we review known results about polytopal realizations, prove polytopality of multi-cluster cluster complexes of rank $2$ (Theorem~\ref{ex:I_2}), and prove that the multi-cluster complex is universal in the sense that every spherical subword complex is the link of a face of a multi-cluster complex (Theorem~\ref{thm:universal}). Section \ref{sec:sufficient} contains a combinatorial description of the sorting words of the longest element of finite Coxeter groups (Theorem~\ref{thm:charac_phi}), and an alternative definition of multi-cluster complexes in terms of the strong intervening neighbors property (Theorem~\ref{thm:if_part_conj}).
In Section~\ref{sec:cyclic action}, we connect multi-cluster complexes to Auslander-Reiten quivers and repetition quivers, and use this connection to introduce an action on vertices and facets of multi-cluster complexes generalizing the natural rotation action on multi-triangulations in type~$A$.
Finally, in Section~\ref{sec:open problems}, we discuss open problems and questions arising in the context of multi-cluster complexes.

\vspace*{10pt}

In a subsequent paper, C.~Stump and V.~Pilaud study the geometry of subword complexes and use the theory developed in the present paper to describe the connections to \emph{Coxeter-sortable elements}, and how to recover \emph{Cambrian fans}, \emph{Cambrian lattices}, and the \emph{generalized associahedra} purely in terms of subword complexes~\cite{pilaud_brick_b_2011}.

%
%

\section{Definitions and main results}\label{sec:mainresults}

In this section, we review the essential notions concerning multi-triangulations, subword complexes and cluster complexes of finite 
type and present the main results. Throughout the paper, $(W,S)$ denotes a \Dfn{finite Coxeter system} of rank $n$, 
and $c$ denotes a \Dfn{Coxeter element}, i.e., the product of the generators in $S$ in some order. The smallest integer $h$ for which $c^h = \Id \in W$~is called \Dfn{Coxeter number}. Coxeter elements of $W$ are in bijection with \emph{(acyclic) orientations of the Coxeter graph} of $W$: a non-commuting pair $s,t \in S$ has the orientation $s \longrightarrow t$ if and only if  \Dfn{$s$ comes before $t$ in~$c$}, i.e., $s$ comes before $t$ in any reduced expression for $c$~\cite{shi_enumeration_1997}. In the simply-laced types $A$, $D$, and $E$, this procedure yields a \Dfn{quiver} $\Omega_c$ associated to a given Coxeter element~$c$, where by quiver we mean a directed graph without loops or two-cycles. For two examples, see Figure~\ref{fig:AR} on page~\pageref{fig:AR}. The \Dfn{length function} on $W$ is given by $\ell(w) = \min\{ r : w = a_{1} \cdots a_{r},\ a_i\in S \}$. An expression for~$w$ of minimal length is called \Dfn{reduced}. The unique \Dfn{longest element} in $W$ is denoted by $\wo$, its length is given by $\ell(\wo) = N := nh/2$. We refer the reader to~\cite{humphreys_reflection_1992} for further definitions and a detailed introduction to finite Coxeter groups. Next, we adopt some writing conventions: in order to emphasize the distinction between words and group elements, we write words in the alphabet $S$ as a sequence between brackets $(a_1,a_2,\dots,a_r)$ and use square letters such as $\mathbf{w}$ to denote them, and we write group elements as a concatenation of letters $a_1a_2\cdots a_r$ using normal script such as~$w$ to denote them.

\subsection{Multi-triangulations} \label{sec:multi-triangulations}

Let $\Dm$ be the simplicial complex with vertices being diagonals of a convex $m$-gon and faces being subsets of non-crossing diagonals. Its facets correspond to \Dfn{triangulations} (i.e., maximal subsets of diagonals which are mutually non-crossing). This simplicial complex is the boundary complex of the \Dfn{dual associahedron}~\cite{haiman_associahedron_1984,lee_associahedron_1989,gelfand_discriminants_2008,loday_realization_2004,hohlweg_realizations_2007,ceballos_many_2011},
we refer to the recent book~\cite{mueller_associahedra_2012} 
for a detailed treatment of the history of associahedra.
The complex~$\Dm$ can be generalized using a positive integer $k$ with $2k +1 \leq m$: define  a \Dfn{$(k+1)$-crossing} to be a set of $k+1$ diagonals which are pairwise crossing. A diagonal is called \Dfn{$k$-relevant} if it is contained in some $(k+1)$-crossing, that is, if there are at least $k$ vertices of the $m$-gon on each side of the diagonal. The complex $\Dmk$ is the simplicial complex of $(k+1)$-crossing free sets of $k$-relevant diagonals. Its facets are given by \Dfn{$k$-triangulations} (i.e., maximal subsets of diagonals which do not contain a $(k+1)$-crossing), without considering $k$-irrelevant diagonals. The reason for restricting the set of diagonals is that including all diagonals that are not $k$-relevant would yield the join of $\Dmk$ and an $mk$-simplex. This simplicial complex has been studied by several authors, see e.g. \cite{dress_line_2003, 
jonsson_generalized_2005, jonsson_spherical_2007, krattenthaler_growth_2006, nakamigawa_generalization_2000, rubey_increasing_2011,stump_new_2011}; an interesting recent treatment of $k$-triangulations using complexes of star polygons can be found in~\cite{pilaud_multitriangulations_2009}.
We refer to this simplicial complex as the simplicial complex of multi-triangulations.

In \cite{stump_new_2011}, the following description of $\Dmk$ is exhibited: let $\S{n+1}$ be the symmetric group generated by the $n$ simple transpositions $s_i = (i \ i+1)$ for $1 \leq i \leq n$, where $n=m-2k-1$. The $k$-relevant diagonals of a convex $m$-gon are in bijection with (positions of) letters in the word
\[
  Q= (\underbrace{s_n, \dots, s_1, \quad \cdots \quad s_n, \dots, s_1}_{k \text{ times } s_n,\ldots,s_1}, 
  \quad s_n, \dots, s_1, \quad s_n, \dots, s_2, \quad \cdots \quad s_n, s_{n-1}, \quad s_n)
\]
of length $kn+\binom{n+1}{2} = \binom{m}{2}-mk$. If the vertices of the $m$-gon are cyclically labeled by the integers from 0 to $m-1$, 
the bijection sends the $i$-th letter of $Q$ to the $i$-th $k$-relevant diagonal in lexicographic order. Under this bijection, a collection of $k$-relevant diagonals forms a facet of $\Dmk$ if and only if the complement of the corresponding subword in $Q$ forms a 
reduced expression for the permutation $[n+1,\ldots,2,1]~\in~\S{n+1}$. A similar approach which admits various possibilities for the 
word $Q$ was described in \cite{pilaud_multitriangulations_2010} in the context of sorting networks. 
We present these general bijections in the context of this paper in Section~\ref{sec:results}.
    
\begin{example}\label{ex:pentagon}
  For $m=5$ and $k=1$, we get $Q = (q_1, q_2, q_3, q_4, q_5) = (s_2,s_1,s_2,s_1,s_2)$. By cyclically labeling the vertices of the pentagon   with the integers $\{0, \dots, 4\}$, the bijection sends the (position of the) letter $q_i$ to the $i$-th entry of the list of ordered diagonals $[0,2], [0,3], [1,3], [1,4], [2,4]$. On one hand, two cyclically consecutive diagonals in the list form a triangulation of the 
  pentagon. On the other hand, the complement of two cyclically consecutive letters of $Q$ form a reduced expression for $[3,2,1] = s_1 s_2 s_1 = s_2 s_1 s_2 \in \S{3}$.
\end{example}
    
The main objective of this paper is to describe and study a natural generalization of multi-triangulations to finite Coxeter groups.

\subsection{Subword complexes}\label{sec:subword complexes}

Let $Q=(q_1,\dots,q_r)$ be a word in the generators $S$ of $W$ and let $\pi \in W$. The subword complex $\Delta(Q,\pi)$ 
was introduced by A.~Knutson and E.~Miller in order to study Gr\"obner geometry of Schubert varieties, see 
\cite[Definition 1.8.1]{knutson_grobner_2005}, and was further studied in \cite{knutson_subword_2004}. It is defined as the simplicial 
complex whose faces are given by subwords $P$ of~$Q$ for which the complement $Q \setminus P$
 contains a reduced expression of $\pi$. Note that subwords come with their embedding into $Q$; two subwords $P$ and $P'$ representing the same word are considered to be different if they involve generators at different positions within $Q$. In Example~\ref{ex:pentagon}, 
 we have seen an instance of a subword complex with $Q=(q_1,q_2,q_3,q_4,q_5) = (s_2,s_1,s_2,s_1,s_2)$ and $\pi=s_1 s_2 s_1 = s_2 s_1 s_2$. In this case, $\Delta(Q,\pi)$ has vertices $\{q_1, \dots, q_5\}$ and facets
\[
\{q_1,q_2\},\{q_2,q_3\},\{q_3,q_4\},\{q_4,q_5\},\{q_5,q_1\}.
\]
Subword complexes are known to be vertex-decomposable and hence shellable~\cite[Theorem~2.5]{knutson_subword_2004}. Moreover, they are topologically spheres or balls depending on the Demazure product of~$Q$.
Let $Q'$ be the word obtained by adding $s \in S$ at the end of a word $Q$. The \Dfn{Demazure product}~$\delta(Q')$ is recursively defined by
\[
\delta(Q') =
\begin{cases}
  \mu s &\text{ if } \ell(\mu s) > \ell(\mu),\\
  \mu &\text{ if } \ell(\mu s) < \ell(\mu),\\
\end{cases}
\]
where $\mu = \delta(Q)$ is the Demazure product of $Q$, and where the Demazure product of the empty word is defined to be the identity element 
in $W$. A subword complex $\Delta(Q,\pi)$ is a sphere if and only if $\delta(Q) = \pi$, and a ball otherwise \cite[Corollary~3.8]{knutson_subword_2004}.    

\subsection{Cluster complexes} \label{sec:cluster complexes}

In \cite{fomin_y-systems_2003}, S.~Fomin and A.~Zelevinsky introduced \Dfn{cluster complexes} associated to finite crystallographic 
root systems. This simplicial complex along with the \Dfn{generalized associahedron} has become the object of intensive studies and  generalizations in various contexts in mathematics, see for instance \cite{chapoton_polytopal_2002, marsh_generalized_2003, reading_clusters_2007,hohlweg_permutahedra_2011}. 
A generator $s \in S$ is called \Dfn{initial} or \Dfn{final} in a Coxeter element $c$ if $\ell(sc) < \ell(c)$ or $\ell(cs) < \ell(c)$, respectively. 
The group~$W$ acts naturally on the real vector space $V$ with basis $\Delta = \{ \alpha_s : s \in S \}$, whose elements are called 
\Dfn{simple roots}.
Let $\Phi$ denote a \Dfn{root system} for $W$, and let $\Phi^+ \subseteq \Phi$ be the set of \Dfn{positive roots} for the simple system~$\Delta$.
Furthermore, let $\Phipm = \Phi^+ \cup -\Delta$ be the set of \Dfn{almost positive roots}.
We denote by $W_{\langle s\rangle}$ the maximal standard parabolic subgroup generated by $S \setminus \{s\}$, and by $\Phi_{\langle s \rangle}$ the associated 
subroot system. For $s \in S$, the involution $\sigma_s : \Phipm \longrightarrow \Phipm$ is given by
\[
  \sigma_s(\beta) =
  \begin{cases}
    \beta &\text{ if } -\beta \in \Delta \setminus \{ \alpha_s \}, \\
    s(\beta) &\text{ otherwise}.\\
  \end{cases}
\]
N.~Reading showed that the definition of cluster complexes can be extended to all finite root systems and enriched with a parameter $c$ being a Coxeter element~\cite{reading_clusters_2007}.
These $c$-cluster complexes are defined using a family~$\parallel_c$ of $c$-compatibility 
relations on~$\Phipm$, see~\cite[Section~5]{reading_sortable_2011}. This family is characterized by the following two properties:
\begin{enumerate}[(i)]
  \item for $s \in S$ and $\beta \in \Phipm$,
    \[-\alpha_s \parallel_c \beta \Longleftrightarrow \beta \in \left(\Phi_{\langle s \rangle}\right)_{\geq -1},\]
  \item for $\beta_1, \beta_2 \in \Phipm$ and $s$ being initial in $c$,
    \[\beta_1 \parallel_c \beta_2 \Longleftrightarrow \sigma_s(\beta_1) \parallel_{scs} \sigma_s(\beta_2).\]
\end{enumerate}    
A maximal subset of pairwise $c$-compatible almost positive roots is called $c$\Dfn{-cluster}. The \Dfn{$c$-cluster complex} 
is the simplicial complex whose vertices are the almost positive roots and whose facets are $c$-clusters. It turns out that all $c$-cluster complexes for the various Coxeter elements are isomorphic, see \cite[Proposition 4.10]{marsh_generalized_2003} and \cite[Proposition 7.2]{reading_clusters_2007}. In crystallographic types, they are moreover isomorphic to the cluster complex as defined in \cite{fomin_y-systems_2003}.

\subsection{Main results}\label{sec:results}

We are now in the position to state the main results of this paper and to define the central object, the \emph{multi-cluster complex}. Let $\c = (c_1, \ldots, c_n)$ be a reduced expression for a Coxeter element $c \in W$, and let $\woc = (w_1, \dots, w_N)$ be the lexicographically first subword of $\c^{\infty}$ which represents a reduced expression for the longest element $\wo \in W$. The word $\woc$ is called \Dfn{$c$-sorting word} for $\wo$, see~\cite{reading_clusters_2007}. The first theorem (proved in Section~\ref{sec:maintheorem}) gives a description of the cluster complex as a subword complex.
            
\begin{theorem}\label{th:main theorem}
  The subword complex $\Delta(\cwoc,\wo)$ is isomorphic to the $c$-cluster complex.   The isomorphism is given by 
  sending the letter $c_i$ of $\c$ to the negative root $-\alpha_{c_i}$,   and the letter $w_i$ of~$\woc$ to the positive 
  root $w_{1} \cdots w_{i-1}(\alpha_{w_i})$.
\end{theorem}

As an equivalent statement, we obtain the following explicit description of the $c$-compatibility relation.

\begin{corollary}\label{cor:maintheorem}
  A subset $C$ of $\Phipm$ is a $c$-cluster if and only if the complement of the corresponding   subword in $\cwoc= 
  (c_1, \ldots, c_n, w_1, \dots, w_N)$ represents a reduced expression for $\wo$. 
\end{corollary}

This description was obtained independently by K.~Igusa and R.~Schiffler \cite{igusa_exceptional_2010} for finite crystallographic root systems in the context of cluster categories \cite[Theorem~2.5]{igusa_exceptional_2010}. They use results of W.~Crawley-Beovey and C.M.~Ringel saying that the braid group acts transitively on isomorphism classes of exceptional sequences of modules over a hereditary algebra, see~\cite[Section~2]{igusa_exceptional_2010}. K.~Igusa and R.~Schiffler then show combinatorially that the braid group acting on sequences of elements in any Coxeter group $W$ of rank $n$ acts as well transitively on all sequences of $n$ reflections whose product is a given Coxeter element~\cite[Theorem~1.4]{igusa_exceptional_2010}. They then deduce Corollary~\ref{cor:maintheorem} in crystallographic types from these two results, see \cite[Theorem~2.5]{igusa_exceptional_2010}. The present approach holds uniformly for all finite Coxeter groups, and is developed purely in the context of Coxeter group theory. We study the connections to the work of K.~Igusa and R.~Schiffler more closely in Section~\ref{sec:cyclic action}.
In the particular case of bipartite Coxeter elements, as defined in
Section~\ref{sec:ART} below, a similar description as in Corollary~\ref{cor:maintheorem} was as
well obtained by T.~Brady and C.~Watt in~\cite{BW2008} in the context
of the geometry of noncrossing partitions\footnote{We thank an
anonymous referee for pointing us to this result.}.

\begin{example}\label{ex:B2}
  Let $W$ be the Coxeter group of type $B_2$ generated by $S=\{s_1,s_2\}$ and let $c = c_1c_2 = s_1 s_2$. Then the word $\cwoc$ is $(c_1, c_2, w_1, w_2, w_3, w_4) = (s_1, s_2, s_1, s_2, s_1, s_2)$. The corresponding list of almost positive roots is 
  \[ 
    [-\alpha_1,\ -\alpha_2,\ \alpha_1,\ \alpha_1+\alpha_2,\ \alpha_1+2\alpha_2,\ \alpha_2].
  \] 
  The subword complex $\Delta(\cwoc,\wo)$ is an hexagon with facets being any two cyclically consecutive letters. The corresponding $c$-clusters are
  \[
    \{-\alpha_1,-\alpha_2\},\{-\alpha_2,\alpha_1\},\{\alpha_1,\alpha_1+\alpha_2\},\{\alpha_1+\alpha_2,\alpha_1+2\alpha_2\},\{\alpha_1+2\alpha_2,\alpha_2\},\{\alpha_2,-\alpha_1\}.
  \]
\end{example}

Inspired by results in \cite{stump_new_2011} and \cite{pilaud_multitriangulations_2010}, we generalize the subword complex in 
Theorem~\ref{th:main theorem} by considering the concatenation of $k$ copies of the word $\c$. In type~$A$, 
this generalization coincides with the description of the complex $\Dmk$ given in~\cite{pilaud_multitriangulations_2010} in a different language.

\begin{definition}\label{def:multi-cluster complex}
  The \Dfn{multi-cluster complex} $\DWk$ is the subword complex $\Delta(\ckwoc, \wo)$.
\end{definition}

Multi-cluster complexes are in fact independent of the Coxeter element $c$. In particular, we reobtain that all $c$-cluster complexes are isomorphic (see Section~\ref{sec:sorting} for the proof). 

\begin{theorem}\label{theorem:independent_of_c}
  All multi-cluster complexes $\DWk$ for the various Coxeter elements are isomorphic.
\end{theorem}

The following two results give alternative descriptions of multi-cluster complexes.
A word $Q=(q_1,\dots, q_r)$ in $S$ has the \Dfn{intervening neighbors property}, if all non-commuting pairs $s,t \in S$ alternate within $Q$, see~\cite[Section~3]{eriksson_conjugacy_2009} and \cite[Proposition 2.1]{speyer_powers_2009}. Let $\psi:S\rightarrow S$ be the involution given by $\psi(s)=\wo^{-1} s \wo$, and extend~$\psi$ to words as $\psi(Q) = (\psi(q_1),\ldots,\psi(q_r))$. We say that $Q$ has the \Dfn{strong intervening neighbors property} (\Dfn{SIN-property}), if $Q\psi(Q)=(q_1,\dots, q_r, \psi(q_1),\dots, \psi(q_r))$ has the intervening neighbors property, and if in addition the Demazure product $\delta(Q)$ is $\wo$. Two words coincide \Dfn{up to commutations} if they can be obtained from each other by a sequence of interchanges of consecutive commuting letters. The next theorem (proved in Section~\ref{sec:sufficient}) characterizes all words that are equal to $\ckwoc$ up to commutations.

\begin{theorem}\label{thm:if_part_conj}
A word in $S$ has the SIN-property if and only if it is equal to $\ckwoc$, up to commutations, for some Coxeter element $c$ and some non-negative integer $k$. 
\end{theorem}

The following proposition gives a different description of the facets
of the multi-cluster complex. It generalizes results
in~\cite[Section~8]{BW2008} (see
also~\cite[Section~2.6]{athanasiadis_$h$-vectors_2006}), and as well
in~\cite[Lemma~3.2]{igusa_exceptional_2010}. In~\cite{BW2008}, the authors consider
the case $k=1$ with bipartite Coxeter elements (see Section~\ref{sec:ART} for
the definition). In~\cite{igusa_exceptional_2010}, the authors consider the case $k=1$
for crystallographic types with arbitrary Coxeter elements.
Set $\ckwoc=(q_1,q_2,\ldots,q_{kn+N})$. For an index $1 \leq i \leq kn+N$, set the reflection $t_i$ to be $q_1q_2\ldots q_{i-1}q_iq_{i-1}\ldots q_2q_1$. E.g., in Example~\ref{ex:B2}, we obtain the sequence
$$(t_1,t_2,t_3,t_4,t_5,t_6) = (s_1,s_1s_2s_1,s_2s_1s_2,s_2,s_1,s_1s_2s_1).$$
\begin{proposition}
  A collection $\{q_{\ell_1},\ldots,q_{\ell_{kn}}\}$ of letters in $\ckwoc$ forms a facet of $\DWk$ if and only if
  $$t_{\ell_{kn}}\cdots t_{\ell_2}t_{\ell_1} = c^k.$$
\end{proposition}
\begin{proof}
  The proof follows the lines of the proof of~\cite[Lemma~3.2]{igusa_exceptional_2010}. A direct calculation shows that $t_{\ell_1}\cdots t_{\ell_{kn}} q_1q_2\cdots q_{kn+N}$ equals the product of all letters in $\cwoc$ not in $\{q_{\ell_1},\ldots,q_{\ell_{kn}}\}$. We get that $\{q_{\ell_1},\ldots,q_{\ell_{kn}}\}$ is a facet of $\DWk$ if and only if $t_{\ell_1}\cdots t_{\ell_{kn}} q_1q_2\cdots q_{kn+N} = \wo$. As $q_1q_2\cdots q_{kn+N} = c^k \wo$, the statement follows.
\end{proof}

We have seen in Section~\ref{sec:multi-triangulations} that the multi-cluster complex of type $A_{m-2k-1}$ is isomorphic to the simplicial complex whose facets correspond to $k$-triangulations of a convex $m$-gon,
\[
  \Dk{A_{m-2k-1}} \cong \Delta_{m,k}.
\]
Thus, the multi-cluster complex extends the concept of multi-triangulations to finite Coxeter groups and provides a unifying approach to multi-triangulations and cluster complexes. 

\begin{remark}\label{rem:naive_multi_def}
Note that there is as well a \lq\lq naive\rq\rq\ way of extending the notion of cluster complexes. 
Consider the simplicial complex on the set of almost positive roots whose faces are given by the sets which do not contain any subset of $k+1$ pairwise not compatible roots. In type~$A$, this complex gives rise to the simplicial complex of multi-triangulations of a convex polygon as desired.
However, this simplicial complex lacks basic properties of cluster
complexes in general; in type $B_3$, it is not pure. In this case, 
the maximal faces have cardinality~$6$ or~$7$. 
A similar phenomenon was observed in \cite[Remark 29]{pilaud_multitriangulations_2010}, where the authors
suggest that subword complexes of type $A$ (viewed as pseudoline
arrangements) are the right objects to define ``multi-pseudotriangulations",
and explain that the approach using pairwise crossings does not work.
\end{remark}

The dictionary for type $A$ is presented in Table \ref{tab:typeA}.
The general bijection between $k$-relevant diagonals of the $m$-gon and (positions of) letters of the word~$Q=\ckwoc$ of type~$A_{m-2k-1}$ is given as follows. Label the vertices of the $m$-gon from~$0$ to~$m-1$ in clockwise direction, and let $n=m-2k-1$ for simplicity. For~$i \in \{ 1,2,\ldots,n\}$, denote by~$p_i$ the position of the generator~$s_i$ in~$\c$, and let 
\begin{align*}
a_i &  = \hspace{1.5cm} \big|\big\{j \in \{ 1,2,\ldots,n\}: j<i \text{ and } p_j<p_{j+1} \big\} \big| \; {\rm mod}(m),\\
b_i & = -k-1- \big|\big\{j \in \{ 1,2,\ldots,n\}: j<i \text{ and } p_j>p_{j+1} \big\} \big| \; {\rm mod}(m).
\end{align*}
The bijection sends the $\ell$-th copy of a generator~$s_i$ in~$Q$ to the $k$-relevant diagonal obtained by rotating~$\ell-1$ times the diagonal $[a_i,b_i]$ in clockwise direction. 
Under this bijection, a collection of $k$-relevant diagonals is a facet of $\Dmk$ if and only if the corresponding subword in $Q$ is a facet of~$\Dk{A_{m-2k-1}}$.

\begin{small}
\begin{table}[!htbp]
\begin{center}
\begin{tabular}{|rll|}
\hline
&&\\[-.13in]
& $\Dmk$ & $\Dk {A_{m-2k-1}}$\\[0.05in]
\hline
&&\\[-.13in]
vertices: & $k$-relevant diagonals of a convex $m$-gon & letters of $Q=\ckwoc$\\[0.05in]
facets: & maximal sets of $k$-relevant diagonals & 
$P\subset Q$ such that $Q\setminus P$ is \\
 & without $(k+1)$-crossings & a reduced expression for $\wo$\\[0.05in]
simplices: & sets of $k$-relevant diagonals  & $P\subset Q$ such that $Q\setminus P$ contains\\
 & without $(k+1)$-crossings & a reduced expression for $\wo$\\[0.05in]
ridges: & flips between two $k$-triangulations & facet flips using Lemma \ref{le:uniquevertex}\\[0.05in]
\hline
\end{tabular}
\vspace{.05in}
\end{center}
\caption{\label{tab:typeA} 
The correspondence between the concepts of diagonals, multi-triangulations and flips of multi-triangulations in $\Dmk$, and the multi-cluster complex $\Dk{A_{m-2k-1}}$.  
}
\vspace{-10pt}
\end{table}
\end{small}

Also in type $B$, we obtain a previously known object, namely the simplicial complex $\Delta_{m,k}^{sym}$ of centrally symmetric $k$-triangulations of a regular convex $2m$-gon. 
The vertices of this complex are pairs of centrally symmetric $k$-relevant diagonals, and a collection of vertices form a face if and only if the corresponding diagonals do not contain a~$(k+1)$-crossing. 
This simplicial complex was studied in algebraic and combinatorial contexts in ~\cite{soll_type-b_2009, rubey_crossings_2010}. 
We refer to Section~\ref{subsec:multi-type-B} for a proof of Theorem~\ref{th:typeB}. 
\begin{theorem} \label{th:typeB}
  The multi-cluster complex $\Dk{B_{m-k}}$ is isomorphic to the simplicial complex of centrally symmetric $k$-triangulations of a regular convex $2m$-gon.
\end{theorem}
The description of the simplicial complex of centrally symmetric multi-triangulations as a subword complex provides straightforward proofs of non-trivial results about centrally symmetric multi-triangulations.

\begin{corollary}\label{cor:sym_multi}
The following properties of centrally symmetric multi-triangulations of a regular convex $2m$-gon hold.
\begin{enumerate}[(i)]
  \item All centrally symmetric $k$-triangulations of a regular convex $2m$-gon contain exactly $mk$ relevant (centrally) symmetric pairs of diagonals, of which $k$ are diameters.
  \item For any centrally symmetric $k$-triangulation $T$ and any $k$-relevant symmetric pair of diagonals $d\in T$ , there exists a unique $k$-relevant symmetric pair of diagonals $d'$ not in $T$ such that $T'=(T\backslash \{d\}) \cup \{d'\}$ is again a centrally symmetric $k$-triangulation. The operation of interchanging a symmetric pair of diagonals between $T$ and $T'$ is called \Dfn{symmetric flip}.
  \item All centrally symmetric $k$-triangulations of a $2m$-gon are connected by symmetric flips.
\end{enumerate}
\end{corollary}
The dictionary between the type $B$ multi-cluster complex and the simplicial complex of centrally symmetric $k$-triangulations of a regular convex $2m$-gon is presented in Table~\ref{tab:typeB}.
\begin{small}
\begin{table}[!htbp]
\begin{center}
\begin{tabular}{|rll|}
\hline
&&\\[-.13in]
& $\Delta_{m,k}^{sym}$ & $\Dk{B_{m-k}}$\\[0.05in]
\hline
&&\\[-.13in]
vertices: & $k$-relevant symmetric pairs of & letters of $Q=\ckwoc=\c^m$\\
& diagonals of a regular convex $2m$-gon & \\[0.05in]
facets: & maximal sets of $k$-relevant centrally symmetric & $P\subset Q$ such that $Q\setminus P$ is\\
 & diagonals without $(k+1)$-crossings & a reduced expression for $\wo$\\[0.05in]
simplices: & sets of $k$-relevant symmetric pairs of & $P\subset Q$ such that $Q\setminus P$ contains\\
 & diagonals without $(k+1)$-crossings & a reduced expression for $\wo$\\[0.05in]
ridges: & symmetric flips between two centrally symmetric & facet flips using Lemma \ref{le:uniquevertex}\\
& $k$-triangulations  &\\[0.05in]
\hline
\end{tabular}
\vspace{.05in}
\end{center}
\caption{\label{tab:typeB} The generalization of the concept of diagonals, multi-triangulations and flips of multi-triangulations to the Coxeter group of type $B_n$.}
\vspace{-10pt}
\end{table}
\end{small}

The bijection between $k$-relevant symmetric pairs of diagonals of a regular convex $2m$-gon and (positions of) letters of the word~$Q=\ckwoc= \c^m$ of type~$B_{n}$, where $n=m-k$ and $(s_1s_2)^4=(s_is_{i+1})^3={\bf 1}$ for $1<i<n$, is given as follows. Label the vertices of the $2m$-gon from~$0$ to~$2m-1$ in clockwise direction. For~$i \in[n]$, denote by~$p_i$ the position of the generator~$s_i$ in~$\c$, and let 
\begin{align*}
a_i &  = \hspace{0.74cm} |\{j \in [n]: j<i \text{ and } p_j<p_{j+1} \}|,\\
b_i & = m - |\{j \in [n]: j<i \text{ and } p_j>p_{j+1}\}|.
\end{align*}
The bijection sends the $\ell$-th copy of a generator~$s_i$ in~$Q$ to the $k$-relevant symmetric pair of diagonals obtained by rotating~$\ell-1$ times the symmetric pair 
$\symdiag{a_i,b_i} :=\big\{[a_i,b_i],[a_i+m,b_i+m]\big\}$ in clockwise direction (observe that both diagonals coincide for $i=1$). 
Under this bijection, a collection of $k$-relevant symmetric pairs of diagonals is a facet of $\Delta_{m,k}^{sym}$ if and only if the corresponding subword in $Q$ is a facet of~$\Dk{B_{m-k}}$.

\begin{example}\label{ex:B3}
Let $m=5$ and $k=2$, and let $W$ be the Coxeter group of type $B_3$ generated by $S=\{s_1,s_2,s_3\}$ where $(s_1s_2)^4=(s_2s_3)^3=(s_1s_3)^2=\Id$. The multi-cluster complex $\Delta_c^2(B_3)$ is isomorphic to the simplicial complex of centrally symmetric $2$-triangulations of a regular convex $10$-gon. In the particular case where the Coxeter element $c=c_1c_2c_3=s_1s_2s_3$, the bijection between $2$-relevant symmetric pairs and the letters of the word $Q=\c^2\woc=(s_1,s_2,s_3)^{5}$ is given by
\begin{scriptsize}
  $$
    \begin{array}{cccccccccccccccccccc}
      s_1 & s_2 & s_3 & s_1 & s_2 & s_3 & s_1 & s_2 & s_3 & s_1 & s_2 & s_3 & s_1 & s_2 & s_3  \\[5pt]
      {[0,5]} & {[1,5]} & {[2,5]} & {[1,6]} & {[2,6]} & {[3,6]} & {[2,7]} & {[3,7]} & {[4,7]} & {[3,8]} & {[4,8]} & {[5,8]} & {[4,9]} & {[5,9]} & {[6,9]} \\
      & {[6,0]} & {[7,0]} & & {[7,1]} & {[8,1]} & & {[8,2]} & {[9,2]} & & {[9,3]} & {[0,3]} & & {[0,4]} & {[1,4]}
    \end{array}.
  $$
\end{scriptsize}
For instance, the first appearance of the letter $s_3$ is mapped to the symmetric pair of diagonals $\symdiag{2,5} = \big\{[2,5],[7,0]\big\}$, while the third appearance of $s_1$ is mapped to the symmetric pair of diagonals $\symdiag{2,7} = \big\{[2,7]\big\}$. The centrally symmetric $k$-triangulations can be easily described using the subword complex approach. For example, the symmetric pairs of diagonals at positions $\{3,5,7,9,13,15\}$ form a facet of $\Delta_{m,k}^{sym}$, and the symmetric flips are interpreted using Lemma~\ref{le:uniquevertex}.
\end{example}

Using algebraic techniques, D.~Soll and V.~Welker proved that $\Delta_{m,k}^{sym}$ is a (mod~$2$)-homology-sphere \cite[Theorem~10]{soll_type-b_2009}. 
The subword complex description in 
Theorem \ref{th:typeB} and the results by A.~Knutson and E.~Miller~\cite[Theorem 2.5 and Corollary 3.8]{knutson_subword_2004} imply the following stronger result.
\begin{corollary}\label{cor:typeBvdecomp}
The simplicial complex of centrally symmetric k-triangulations of a regular convex $2m$-gon is a vertex-decomposable simplicial sphere.
\end{corollary}
This result together with the proof of \cite[Conjecture 13]{soll_type-b_2009} given in \cite{rubey_crossings_2010}\footnote{The proof appeared in Section 7 in the arxiv version, see {\tt http://arxiv.org/abs/0904.1097v2}.} implies the following conjecture by D.~Soll and V.~Welker.
\begin{corollary}[{\cite[Conjecture 17]{soll_type-b_2009}}]\label{cor:conj17}
For the term-order $\preceq$ defined in \cite[Section~7]{soll_type-b_2009}, the initial ideal $\operatorname{in}_\preceq(I_{n,k})$ of the determinantal ideal $I_{n,k}$ defined in \cite[Section~3]{soll_type-b_2009} is spherical.
\end{corollary}

We finish this section by describing all spherical subword complexes in terms of faces of multi-cluster complexes (see Section~\ref{subsec:Genmultiasso} for the proofs). This generalizes the universality of the multi-associahedron presented in~\cite[Proposition~5.6]{pilaud_brick_2011} to finite Coxeter groups.

\begin{theorem}\label{thm:universal}
  A simplicial sphere can be realized as a subword complex of a given finite type $W$ if and only if it is the link of a face of a multi-cluster 
  complex $\DWk$. 
\end{theorem}

The previous theorem can be obtained for any family of subword complexes, for which arbitrary large powers of $\c$ appear as subwords. However, computations seem to indicate that the multi-cluster complex maximizes the number of facets among subword complexes $\Delta(Q,\wo)$ with word $Q$ of the same size. We conjecture that this is true in general, see Conjecture~\ref{conj:maximal}. We also obtain the following corollary.

\begin{corollary}\label{cor:polytopality}
  The following two statements are equivalent.
  \begin{enumerate}[(i)]
    \item Every spherical subword complex of type~$W$ is polytopal.
    \item Every multi-cluster complex of type~$W$ is polytopal.
  \end{enumerate}
\end{corollary}

%
%

\section{General results on spherical subword complexes} \label{sec:spherical subword complexes}

Before proving the main results, we discuss several properties of spherical subword complexes in general which are not specific to multi-cluster complexes. Throughout this section, we let  $Q = (q_{1}, \dots, q_r)$ be a word in $S$ and $\pi = \delta(Q)$.

\subsection{Flips in spherical subword complexes} \label{sec:flip}

\begin{lemma}[Knutson--Miller] \label{le:facets}
  Let $F$ be a facet of $\Delta(Q,\delta(Q))$. For any vertex $q \in F$, there exists a unique vertex $q' \in Q \setminus F$ such that $\big( F \setminus \{ q\}\big) \cup \{q'\}$ is again a facet.
\end{lemma}

\begin{proof}
  This follows from the fact that $\Delta(Q,\delta(Q))$ is a simplicial sphere  \cite[Corollary~3.8]{knutson_subword_2004}.   See~\cite[Lemma~3.5]{knutson_subword_2004} for an analogous reformulation. 
\end{proof} 

Such a move between two adjacent facets is called \Dfn{flip}. Next, we describe how to find the unique vertex $q' \notin F$ corresponding to $q \in F$. For this, we introduce the notion of root functions.

\begin{definition}\label{def:rootfunction}
The \Dfn{root function} $r_F:Q\rightarrow \Phi$ associated to a facet $F$ of $\Delta(Q,\pi)$ sends a letter $q \in Q$ to the root $\r_F(q):=w_q  (\alpha_q) \in \Phi$, where $w_q  \in W$ is given by the product of the letters in the prefix of $Q\setminus F=(q_{i_1}, \dots, q_{i_\ell})$ that appears on the left of $q$ in $Q$, and where $\alpha_q$ is the simple root associated to $q$.
\end{definition}

\begin{lemma}\label{le:uniquevertex}
  Let $F$, $q$ and $q'$ be as in Lemma~\ref{le:facets}. The vertex $q'$ is the unique vertex not in $F$ for which $\r_F(q') \in \{ \pm \r_F(q) \}$.
\end{lemma}

\begin{proof}
  Since $q_{i_1}\dots q_{i_\ell}$ is a reduced expression for $\pi=\delta(Q)$, the set $\{\r_F({q_{i_1}}), \dots, \r_F({q_{i_\ell}})\}$ is 
  equal to the inversion set $\inv(\pi) = \{ \alpha_{i_1}, q_{i_1}(\alpha_{i_2}), \ldots, q_{i_1}\cdots q_{i_{\ell-1}}(\alpha_{i_\ell}) \}$ of $\pi$, which only depends on $\pi$ and not on the chosen reduced expression. In particular, any two elements in this set are distinct. Notice that the root $\r_F(q)$ 
  for $q \in F$ is, up to sign, also contained in $\inv(\pi)$, otherwise it would contradict the fact that the Demazure product of $Q$ is $\pi$. 
  If we insert $q$ into the reduced expression of $\pi$, the exchange property in Coxeter groups implies that we have to delete the unique letter $q'$ that corresponds to the same root, with a positive sign if it appears on the right of $q$ in $Q$, or with a negative sign otherwise. The resulting word is again a reduced expression for $\pi$.
\end{proof}

\begin{remark}
  In the case of cluster complexes, this description can be found in~\cite[Lemma~2.7]{igusa_exceptional_2010}.
\end{remark}

\begin{example}\label{ex:root function}
  As in Example \ref{ex:B2}, consider the Coxeter group of type $B_2$ generated by $S=\{s_1,s_2\}$ with $c = c_1c_2 = s_1 s_2$ and 
  $\cwoc=(c_1, c_2, w_1, w_2, w_3, w_4) = (s_1, s_2, s_1, s_2, s_1, s_2)$. Consider the facet $F=\{c_2,w_1\}$, we obtain
  \begin{align*}
    \r_F(c_1) &= \alpha_1, & \r_F(w_2) &= s_1(\alpha_2) = \alpha_1+\alpha_2, \\
    \r_F(c_2) &= s_1(\alpha_2) = \alpha_1+\alpha_2, & \r_F(w_3) &= s_1 s_2(\alpha_1) = \alpha_1+2\alpha_2,\\
    \r_F(w_1) &= s_1(\alpha_1) = -\alpha_1, & \r_F(w_4) &= s_1 s_2 s_1 (\alpha_2) = \alpha_2. 
  \end{align*}
  Since $\r_F(c_2)=\r_F(w_2)$,  the letter $c_2$ in $F$ flips to $w_2$. As $w_2$ appears on the right of $c_2$, both roots have the same 
  sign. Similarly, the letter $w_1$ flips to $c_1$, because $\r_F(c_1)=-\r_F(w_1)$. In this case, the roots have different signs because 
  $c_1$ appear on the left of $w_1$.
\end{example}

The following lemma describes the relation between the root functions of two facets connected by a flip.

\begin{lemma}\label{le:flip}
  Let $F$ and $F' = \big(F \setminus \{q\}\big)\cup\{q'\}$ be two adjacent facets of the subword complex $\Delta(Q,\delta(Q))$, and assume that $q$ appears on the left of $q'$ in Q. Then, for every letter $p\in Q$,
  \[
    \r_{F'}(p) = 
    \begin{cases}
      t_q(\r_F(p)) & \text{if } p \text{ is between } q \text{ and } q' \text{, or } p=q', \\
      \r_F(p) & \text{otherwise.}
    \end{cases}
  \]
  Here, $t_q=w_qqw_q^{-1}$ where $w_q$ is the product of the letters in the prefix of $Q\backslash F$ that appears on the left of $q$ in $Q$. By construction, $t_q$ is the reflection in $W$ orthogonal to the root $r_F(q)=w_q(\alpha_q)$.
\end{lemma}

\begin{proof}
Let $p$ be a letter in $Q$, and $w_p ,w_p '$ be the products of the letters in the prefixes of $Q\backslash F$ and $Q\backslash F'$ that appear on the left of $p$. Then, by definition $r_F(p)=w_p (\alpha_p)$ and $r_{F'}(p)=w_p '(\alpha_p)$. We consider the following three cases:
\begin{itemize}
  \item If $p$ is on the left of $q$ or $p=q$, then $w_p =w_p '$ and $r_F(p)=r_{F'}(p)$.
  \item If $p$ is between $q$ and $q'$ or $p=q'$, then $w_p '$ can be obtained from $w_p $ by adding the letter $q$ at its corresponding position. This addition is the result of multiplying $w_p $ by $t_q=w_qqw_q^{-1}$ on the left, i.e. $w_p '=t_qw_p $. Therefore, $\r_F(p)=t_q(\r_{F'}(p))$.
  \item If $p$ is on the right of $q'$, consider the reflection $t_{q'}=w_{q'}q'w_{q'}^{-1}$ where $w_{q'}$ is the product of the letters in the prefix of $Q\backslash F$ that appears on the left of $q'$. By the same argument, one obtains that $w_p '=t_qt_{q'}w_p $. In addition, $t_q=t_{q'}$ because they correspond to the unique reflection orthogonal to the roots $r_F(q)$ and $r_F(q')$, which are up to sign equal by Lemma~\ref{le:uniquevertex}. Therefore, $w_p '=w_p $ and $r_{F'}(p)=r_F(p)$.  
\end{itemize}
\end{proof}

\subsection{Isomorphic spherical subword complexes}\label{subsec:isomorphic subword complexes}

We now reduce the study of spherical subword complexes in general to the case where $\delta(Q) = \pi = \wo$, and give two operations on the word $Q$ giving isomorphic subword complexes.

\begin{theorem}\label{th:subword complexes for w0}
  Every spherical subword complex $\Delta(Q,\pi)$ is isomorphic to $\Delta(Q',\wo)$, for some word $Q'$ such that $\delta(Q')=\wo$.
\end{theorem}

\begin{proof}
  Let $\mathbf{r}$ be a reduced word for $\pi^{-1}\wo=\delta(Q)^{-1}\wo \in W$. Moreover, define the word $Q'$ as the 
  concatenation of $Q$ and $\mathbf{r}$. By construction, the Demazure product of $Q'$ is $\wo$, and every reduced expression 
  of $\wo$ in $Q'$ must contain all the letters in $\mathbf{r}$. The reduced expressions of $\wo$ 
  in $Q'$ are given by reduced expressions of $\pi$ in $Q$ together with all the letters in $\mathbf{r}$. Therefore, the subword 
  complexes $\Delta(Q,\pi)$ and $\Delta(Q',\wo)$ are isomorphic.
\end{proof}

Recall the involution $\psi:S\rightarrow S$ given by $\psi(s)=\wo^{-1} s \wo$. This involution was used in~\cite{bergeron_isometry_2009} to characterize isometry classes of the $c$-generalized associahedra. Define the \Dfn{rotated word} $\rotatedword{Q}{s}$ or the \Dfn{rotation} of $Q = (s,q_2,\ldots,q_r)$ along the letter $s$ as $(q_2, \dots, q_r, \psi(s))$. The following two propositions are direct consequences of the definition of subword complexes.

\begin{proposition}\label{pr:up to commutations}
  If two words $Q$ and $Q'$ coincide up to commutations, then $\Delta(Q,\pi) \cong \Delta(Q',\pi)$.
\end{proposition}
\begin{proof}
  The isomorphism between $\Delta(Q,\pi)$ and $\Delta(Q',\pi)$ is induced by reordering the letters of $Q$ to obtain $Q'$.
\end{proof}

\begin{proposition}\label{pr:rotated}
  Let $Q = (s,q_2,\ldots,q_r)$. Then $\Delta(Q,\wo)\cong \Delta(\rotatedword{Q}{s}, \wo)$.
\end{proposition}
\begin{proof}
  The isomorphism between $\Delta(Q,\wo)$ and $\Delta(\rotatedword{Q}{s},\wo)$ is induced by sending $q_i$ to $q_i$ for $2 \leq i \leq r$ and the initial $s$ to the final $\psi(s)$. The results follows using the fact that $s\wo = \wo\psi(s)$.
\end{proof}

Theorem~\ref{th:subword complexes for w0} and Proposition~\ref{pr:rotated} give an alternative viewpoint on spherical subword complexes. First, we can consider $\pi$ to be the longest element $\wo \in W$. Second, $\Delta(Q,\wo)$ does not depend on the word $Q$ but on the bi-infinite word
\begin{align*}
  \widetilde Q &= \cdots \hspace{15pt} Q \hspace{45pt} \psi(Q) \hspace{40pt} Q \hspace{26pt} \cdots \\
               &= \ldots q_1,\ldots,q_r,\psi(q_1),\ldots,\psi(q_r),q_1,\ldots,q_r,\ldots.
\end{align*}
Taking any connected subword in $\widetilde Q$ of length $r$ gives rise to an isomorphic spherical subword complex.

%
%

\section{Proof of Theorem \ref{theorem:independent_of_c}} \label{sec:sorting}

In this section, we prove that all multi-cluster complexes for the various Coxeter elements are isomorphic. This result relies on the theory of \emph{sorting words} and \emph{sortable elements} introduced by N. Reading in \cite{reading_clusters_2007}. The $c$-\Dfn{sorting word} for $w\in W$ is the lexicographically first (as a sequence of positions) subword of $\c^{\infty}=
\c\c\c\dots$ which is a reduced word for $w$. We use the following result of D. Speyer.

\begin{lemma}[{\cite[Corollary 4.1]{speyer_powers_2009}}]\label{le:speyer}
  The longest element $\wo \in W$ can be expressed as a reduced prefix of $\c^\infty$ up to commutations.
\end{lemma}

The next lemma unifies previously known results; the first statement it trivial, the second statement can be found in \cite[Section 4]{speyer_powers_2009}, and the third statement is equivalent to \cite[Lemma 1.6]{hohlweg_permutahedra_2011}.

\begin{lemma}\label{lem:prefix}
  Let $s$ be initial in $c$ and let ${\bf p}=(s,p_2,\dots, p_r)$ be a prefix of $\c^\infty$ up to commutations. Then,
  \begin{enumerate}[(i)]
    \item $(p_2,\dots,p_r)$ is a prefix of $({\bf c'})^\infty$ up to commutations, where $\bf c'$ denotes a word for the Coxeter element $c' = scs$,
    \item if $p=sp_2\cdots p_r$ is reduced then $\bf p$ is the $c$-sorting word for $p$ up to commutations,
    \item if $sp_2\cdots p_r s'$ is reduced for some $s'\in S$ then $\bf p$ is a prefix of the $c$-sorting word for $ps'$ up to commutations.
  \end{enumerate}
\end{lemma}

\begin{proposition}\label{prop:init_letter}
  Let $s$ be initial in $c$ and let $\woc=(s,w_2,\dots, w_N)$ be the $c$-sorting word of $\wo$ up to commutations. Then, $(w_2,\dots, w_N,\psi(s))$ is the $scs$-sorting word of $\wo$ up to commutations. 
\end{proposition}

\begin{proof}
  By Lemma~\ref{le:speyer}, the element $\wo$ can be written as a prefix of $\c^\infty$. By Lemma~\ref{lem:prefix}, this prefix is equal to the $c$-sorting of $\wo$, which we denote by $\woc$. Let ${\bf scs}$ denote the word for the Coxeter element scs. By Lemma~\ref{lem:prefix} $(i)$, the word $(w_2,\dots,w_N)$ is a prefix of $({\bf scs})^\infty$ and by~$(ii)$ it is the $scs$-sorting word for $w_2\cdots w_N$. By definition of $\psi$, the word $(w_2,\dots,w_N, \psi(s))$ is a reduced expression for $\wo$. Lemma~\ref{lem:prefix} $(iii)$ with the word $(w_2,\dots, w_N)$ and $\psi(s)$ implies that $(w_2,\dots, w_N, \psi(s))$ is the $scs$-sorting word for $\wo$ up to commutations.
\end{proof}

\begin{remark}
  In \cite{reading_sortable_2011}, N.~Reading and D.~Speyer present a uniform approach to the theory of sorting words and sortable elements. This approach uses an anti-symmetric bilinear form which is used to extend many results to infinite Coxeter groups. In particular, the previous proposition can be easily deduced from \cite[Lemma 3.8]{reading_sortable_2011}.
\end{remark}

We are now in the position to prove that all multi-cluster complexes for the various Coxeter elements are isomorphic. 

\begin{proof}[Proof of Theorem~\ref{theorem:independent_of_c}]
  Let $c$ and $c'$ be two Coxeter elements such that $c'=scs$ for some initial letter $s$ of $c$, and let ${\bf c}$ and ${\bf c'}$ denote reduced words for $c$ and $c'$, respectively. Moreover,   let $Q_c=\c^k\woc$, and $Q_{c'} = ({\bf c'})^k\woword({\bf c'})$. By Proposition~\ref{pr:up to commutations}, we can assume that $Q_c = (s, c_2,\dots, c_n)^k \cdot (s, w_2,\dots, w_N)$, and by Proposition \ref{prop:init_letter}, we can also assume that $Q_{c'} = (c_2,\dots, c_n,s)^k \cdot (w_2,\dots, w_N,\psi(s))$. Therefore, 
  $Q_{c'} = \rotatedword{(Q_c)}{s}$, and Proposition~\ref{pr:rotated} implies that the subword complexes $\Delta(Q_c,\wo)$ and $\Delta(Q_{c'},\wo)$ are isomorphic. Since any two Coxeter elements can be obtained from each other by conjugation of initial letters (see \cite[Theorem 3.1.4]{geck_characters_2000}), the result follows.
\end{proof}

%
%

\section{Proof of Theorem~\ref{th:main theorem}} \label{sec:maintheorem}

In this section, we prove that the subword complex $\Delta(\cwoc,\wo)$ is isomorphic to the $c$-cluster complex. As in Theorem~\ref{th:main theorem}, we identify letters in $\cwoc = (c_1,\ldots,c_n,w_1,\ldots,w_N)$ with almost positive roots using the bijection $\Lrc: \cwoc\ \tilde\longrightarrow\ \Phipm$ given by
$$
  \Lrc(q) =
  \begin{cases}
    -\alpha_{c_i} & \text{ if } q=c_i \text{ for some } 1 \leq i \leq n, \\
    w_1 w_2\cdots w_{i-1}(\alpha_{w_i}) & \text{ if } q=w_i \text{ for some } 1 \leq i \leq N.
  \end{cases}
$$
In~\cite{reading_clusters_2007} this map was used to establish a bijection between $c$-sortable elements and $c$-clusters. Note that under this bijection, letters of $\cwoc$ correspond to almost positive roots and subwords of $\cwoc$ correspond to subsets of almost positive roots. We use this identification to simplify several statements in this section. Observe, that 
for the particular facet $F_0$ of $\Delta(\cwoc,\wo)$ corresponding to the prefix $\c$ of $\cwoc$, we have that
$$\Lrc(q) = \r_{F_0}(q) \text{ for every } q \in \woc\subset \cwoc,$$
where $\r_{F_0}(q)$ is the root function as defined in Definition~\ref{def:rootfunction}. We interpret the two parts~$(i)$ and~$(ii)$ in the definition of $c$-compatibility (see Section~\ref{sec:cluster complexes}), in Theorem~\ref{th:1} and Theorem~\ref{th:2}. Proving these two conditions yields a proof of Theorem~\ref{th:main theorem}. The majority of this section is devoted to the proof of the initial condition. The proof of the recursive condition follows afterwards.

\subsection{Proof of condition $(i)$}

The following theorem implies that $\Delta(\cwoc,\wo)$ satisfies the initial condition.

\begin{theorem}\label{th:1}
  $\{-\alpha_s, \beta \}$ is a face of the subword complex $\Delta(\cwoc,\wo)$ if and only if $\beta \in \left(\Phi_{\langle s \rangle}\right)_{\geq -1}$.
\end{theorem}

We prove this theorem in several steps.

\begin{lemma}\label{lem:facet_ci}
Let $F$ be a facet of the subword complex $\Delta(\cwoc,\wo)$ such that $c_i\in F$. Then
\begin{enumerate}[(i)]
\item for every $q\in F$ with $q\neq c_i$, $r_F(q)\in \Phi_{\langle c_i \rangle}$.
\item for every $q \in \cwoc$, $r_F(q)\in \Phi_{\langle c_i \rangle}$ if and only if $\Lrc(q)\in (\Phi_{\langle c_i \rangle})_{\geq -1}$.
\end{enumerate}
\end{lemma}

\begin{proof}
For the proof of $(i)$ notice that if $F=\c$ then the result is clear. 
Now suppose the result is true for a given facet $F$ with $c_i\in F$, and consider the facet $F' = \big(F \setminus \{p\}\big)\cup\{p'\}$ obtained by flipping a letter $p\neq c_i$ in $F$.
Since all the facets containing $c_i$ are connected by flips which do not involve the letter $c_i$, then it is enough to prove the result for the facet $F'$.
By hypothesis, since $p\in F$ and $p\neq c_i$ then $r_F(p)\in \Phi_{\langle c_i \rangle}$. Then, the reflection $t_p$ orthogonal to $r_F(p)$ defined in Lemma \ref{le:flip} satisfies $t_p\in W_{\langle c_i \rangle}$. Using Lemma \ref{le:flip} we obtain that for every $q \in \cwoc$,
\[
r_{F'}(q)\in \Phi_{\langle c_i \rangle} \Leftrightarrow r_F(q) \in \Phi_{\langle c_i \rangle}.
\]   
If $q\in F'$ and $q\neq c_i$ then $(q\in F \text{ and } q\neq c_i)$ or $q=p'$. In the first case, $r_F(q)$ is contained in $\Phi_{\langle c_i \rangle}$ by hypothesis, and consequently $r_{F'}(q) \in \Phi_{\langle c_i \rangle}$. By Lemma \ref{le:uniquevertex}, the second case $q=p'$ implies that $r_F(q)=\pm r_F(p)$. Again since $r_F(p)$ belongs to $\Phi_{\langle c_i \rangle}$ by hypothesis, the root $r_{F'}(q)$ belongs to $\Phi_{\langle c_i \rangle}$.

For the second part of the lemma, notice that the set $\{ q \in \cwoc: r_F(q) \in \Phi_{\langle c_i \rangle} \}$ is invariant for every facet $F$ containing $c_i$. In particular, if $F= \c$ this set is equal to $\{q \in \cwoc: \Lrc(q) \in (\Phi_{\langle c_i \rangle})_{\geq -1} \}$.  Therefore, $r_F(q)\in \Phi_{\langle c_i \rangle}$ if and only if $\Lrc(q) \in (\Phi_{\langle c_i \rangle})_{\geq -1}$.
\end{proof}

\begin{proposition}\label{prop:11}
  If a facet $F$ of $\Delta(\cwoc,\wo)$ contains $c_i$ and $q\neq c_i$, then $\Lrc(q)\in (\Phi_{\langle c_i \rangle})_{\geq -1}$.
\end{proposition}

\begin{proof}
  This proposition is a direct consequence of Lemma~\ref{lem:facet_ci}. 
\end{proof}

Next, we consider the parabolic subgroup $W_{\langle c_i \rangle}$ obtained by removing the generator $c_i$ from $S$.

\begin{lemma}\label{lem:parabolic}
  Let $c'$ be the Coxeter element of the parabolic subgroup $W_{\langle c_i \rangle}$ obtained from $c$ by removing the generator 
  $c_i$. Consider the word $\widehat Q = {\bf c' \wo}(\c)$ obtained by deleting the letter $c_i$ from $Q=\cwoc$, and let 
  $Q'={\bf c'w}_\circ ({\bf c'})$. Then, the subword complexes $\Delta(\widehat Q,\wo)$ and $\Delta(Q',w'_\circ)$ are isomorphic.
\end{lemma}

\begin{proof}
  Since every facet $F$ of $\Delta(\widehat Q, \wo)$ can be seen as a facet $F \cup \{c_i\}$ of $\Delta(\cwoc,\wo)$ which contains $c_i$, then for every $q\in F$ we have that $\Lrc (q)\in (\Phi_{\langle c_i \rangle})_{\geq -1}$ by Proposition \ref{prop:11}. This means that only the letters of $\widehat Q$ that correspond to roots in $(\Phi_{\langle c_i \rangle})_{\ge -1}$ appear in the subword complex $\Delta(\widehat Q,\wo)$. The letters in $Q'$ are in bijection, under the map $\Lrvar{c'}$, with the almost positive roots $(\Phi_{\langle c_i \rangle})_{\ge -1}$. Let $\varphi$ be the map that sends a letter $q\in \widehat Q$ corresponding to a root in $(\Phi_{\langle c_i \rangle})_{\ge -1}$ to the letter in $Q'$ corresponding to the same root. We will prove that $\varphi$ induces an isomorphism between the subword complexes $\Delta(\widehat Q,\wo)$ and $\Delta(Q',w'_\circ)$. In other words, we show that $F$ is a facet of $\Delta(\widehat Q, \wo)$ if and only if $\varphi (F)$ is a facet of $\Delta(Q',\wo ')$. Let $\widetilde \r_F$ and $\r'_{\varphi(F)}$ be the root functions associated to $F$ and $\varphi(F)$ in  $\widehat Q$ and $Q'$ respectively. Then, for every $q\in \widehat Q$ such that $\Lrc(q)\in (\Phi_{\langle c_i \rangle})_{\ge -1}$, we have 
  \begin{align}
    \widetilde \r_F(q)=\r '_{\varphi (F)} (\varphi (q) ). \label{eq:equality}\tag{$\star$}
  \end{align}
  If $F={\bf c'}$ then $\varphi(F)={\bf c'}$ and the equality~\eqref{eq:equality} holds by the definition of $\varphi$. Moreover, if \eqref{eq:equality} holds for a facet $F$ then it is true for a facet $F'$ obtained by flipping a letter in $F$. This follows by applying Lemma \ref{le:flip} and using the fact that the positive roots $(\Phi_{\langle c_i \rangle})_{\ge -1}$ in $\widehat Q$ and $Q'$ appear in the same order, see \cite[Prop. 3.2]{reading_clusters_2007}. Finally, Lemma \ref{le:uniquevertex} and \eqref{eq:equality} imply that the map $\varphi$ sends flips to flips. Since ${\bf c'}$ and $\varphi({\bf c'}) $ are facets of $\Delta(\widehat Q, \wo)$ and $\Delta(Q',\wo ')$ respectively, and all facets are connected by flips, $F$ is a facet of $\Delta(\widehat Q, \wo)$ if and only if $\varphi (F)$ is a facet of $\Delta(Q',\wo ')$. 
\end{proof}

The next lemma states that every letter in $\cwoc$ is indeed a vertex of $\Delta(\cwoc,\wo)$.

\begin{lemma}\label{lem:12}
  Every letter in $\cwoc$ is contained in some facet of $\Delta(\cwoc,\wo)$. 
\end{lemma}

\begin{proof}
  Write the word $Q=\cwoc$ as the concatenation of $\c$ and the $c$-factorization of $\wo$, 
  i.e., $Q= \c \c_{K_1} \c_{K_2} \cdots \c_{K_r}$, where $K_i\subseteq S$ for $1\leq i\leq r$ and $c_I$, with $I\subseteq S$, 
  is the Coxeter element of~$W_I$ obtained from $c$ by keeping only letters in $I$. Since $\wo$ is $c$-sortable, see \cite[Corollary 4.4]{reading_clusters_2007}, 
  the sets $K_i$ form a decreasing chain of subsets of $S$, i.e., $K_r\subseteq K_{r-1} \subseteq \cdots \subseteq 
  K_1 \subseteq S$. This implies that the word $\c \c_{K_1} \ldots \widehat \c_{K_i} \ldots \c_{K_r}$ contains a reduced 
  expression for $\wo$ for any $1\leq i \leq r$. Thus, all letters in $\c_{K_i}$ are indeed vertices.
\end{proof}

\begin{proposition}\label{prop:c_i_and_q}
For every $q \in \cwoc$ satisfying $\Lrc(q) \in (\Phi_{\langle c_i \rangle})_{\geq -1}$, there exists a facet of $\Delta(\cwoc, \wo)$ that contains both $c_i$ and $q$.
\end{proposition}

\begin{proof}
Consider the parabolic subgroup $W_{\langle c_i \rangle}$ obtained by removing the letter $c_i$ from $S$, and let~$\widehat Q$ and $Q'$ be the words as defined in Lemma \ref{lem:parabolic}. Since $\Delta(\widehat Q,\wo)$ and $\Delta(Q',w'_\circ)$ are isomorphic, applying Lemma ~\ref{lem:12} to $\Delta(Q',\wo')$ completes the proof.
\end{proof}

\begin{proof}[Proof of Theorem~\ref{th:1}]
Taking $c_i=s$, $-\alpha_s=\Lrc(c_i)$ and $\beta=\Lrc(q)$ the two directions of the equivalence follow from Propositions~\ref{prop:11} and~\ref{prop:c_i_and_q}.
\end{proof}

\subsection{Proof of condition $(ii)$}

The following theorem proves condition $(ii)$.

\begin{theorem}\label{th:2}
  Let $\beta_1, \beta_2 \in \Phipm$ and $s$ be an initial letter of a Coxeter element $c$. Then,  $\{\beta_1,\beta_2\}$ 
  is a face of the subword complex $\Delta(\cwoc,\wo)$ if and only if $\{\sigma_s(\beta_1),\sigma_s(\beta_2)\}$ 
  is a face of the subword complex $\Delta({\bf c'w}_\circ({\bf c'}),\wo)$, with $c'=scs$.
\end{theorem}
\begin{proof}
  Let $Q=\cwoc$, $s$ be initial in $c$ and $\rotatedword{Q}{s}$ be the rotated word of $Q$, as defined in 
  Section~\ref{subsec:isomorphic subword complexes}.   By Proposition~\ref{prop:init_letter},  the word $\rotatedword{Q}{s}$ is equal to $\mathbf{c'}\woword({\bf c'})$ 
  up to commutations, and by Proposition~\ref{pr:rotated}  the subword complexes $\Delta(\cwoc,\wo)$ and 
  $\Delta({\bf c'w}_\circ({\bf c'}),\wo)$ are isomorphic. For every letter $q\in \cwoc$, we denote by $q'$ the corresponding 
  letter in $\mathbf{c'}\woword({\bf c'})$ obtained from the previous isomorphism. We write $q_1 \sim_c q_2$ if and only if $\{q_1,q_2\}$ is a face of 
  $\Delta(\cwoc,\wo)$. In terms of almost positive roots this is written as
  \[
    \Lrc(q_1) \sim_{c} \Lrc(q_2) \Longleftrightarrow \Lrvar{scs}(q'_1) \sim_{scs} \Lrvar{scs}\left({q'_2}\right).
  \]
  Note that the bijection $\Lrvar{scs}$ can be described using $\Lrc$. Indeed, it is not hard to check that 
  $\Lrvar{scs}(q')=\sigma_s(\Lrc(q))$ for all $q\in Q$. Therefore,
  \[
    \Lrc(q_1) \sim_{c} \Lrc(q_2) \Longleftrightarrow \sigma_s(\Lrc(q_1)) \sim_{scs} \sigma_s(\Lrc(q_2)).
  \]
  Taking $\beta_1=\Lrc(q_1)$ and $\beta_2=\Lrc(q_2)$ we get the desired result.
\end{proof}

%
%

\section{Generalized multi-associahedra and polytopality of spherical subword complexes}\label{sec:polytopality}

In this section, we discuss the polytopality of spherical subword complexes and present what is known in the particular cases  of cluster complexes, simplicial complexes of multi-triangulations, and simplicial complexes of centrally symmetric multi-triangulations. We then prove polytopality of multi-cluster complexes of rank $2$. Finally, we show that every spherical subword complex is the link of a face of a multi-cluster complex, and consequently reduce the question of realizing spherical subword complexes to the question of realizing multi-cluster complexes. We use the term \emph{generalized multi-associahedron} for the dual of a polytopal realization of a multi-cluster complex --~but \emph{the existence of such realizations remains open} in general, see Table~\ref{tab:realization}. The subword complex approach provides new perspectives and methods for finding polytopal realizations. In a subsequent paper, C.~Stump and V.~Pilaud obtain a geometric construction of a class of subword complexes containing generalized associahedra purely in terms of subword complexes \cite{pilaud_brick_b_2011}.

\begin{small}
\begin{table}[!htbp]
\begin{center}
\begin{tabular}{|rl|}
\hline
&\\[-.13in]
simplicial complex generated by & polytopal realization of the dual\\[0.05in]
\hline
&\\[-.13in]
triangulations & associahedron \\
(classical)& \cite{haiman_associahedron_1984,lee_associahedron_1989,gelfand_discriminants_2008,loday_realization_2004,hohlweg_realizations_2007,ceballos_many_2011} \\[0.05in]
multi-triangulations & multi-associahedron \\
\cite{jonsson_generalized_2005, krattenthaler_growth_2006, pilaud_multitriangulations_2009, pilaud_multitriangulations_2010, stump_new_2011} & (existence conjectured) \\[0.05in]
centrally symmetric multi-triangulations & multi-associahedron of type $B$\\
\cite{soll_type-b_2009, rubey_crossings_2010} & (existence conjectured) \\[0.05in]
clusters & generalized associahedron \\
\cite{fomin_y-systems_2003,reading_cambrian_2006,reading_clusters_2007,reading_sortable_2007} & \cite{chapoton_polytopal_2002,hohlweg_realizations_2007,stella_polyhedral_2011,pilaud_brick_b_2011} \\[0.05in]
multi-clusters & generalized multi-associahedron \\
(present paper)&  (existence conjectured)\\[0.05in]
\hline
\end{tabular}
\vspace{.05in}
\end{center}
\caption{\label{tab:realization}Dictionary for generalized concepts of triangulations and associahedra.}
\vspace{-10pt}
\end{table}
\end{small}

\subsection{Generalized associahedra}\label{subsec:genass}

We have seen that for $k=1$, the multi-cluster complex $\Delta_c^1(W)$ is isomorphic to the $c$-cluster complex. S.~Fomin and A.~Zelevinsky conjectured the existence of polytopal realizations of the cluster complex in~\cite[Conjecture~1.12]{fomin_y-systems_2003}. F.~Chapoton, S.~Fomin, and A.~Zelevinsky then proved this conjecture by providing explicit inequalities for the defining hyperplanes of generalized associahedra~\cite{chapoton_polytopal_2002}. N.~Reading constructed $c$-Cambrian fans, which are complete simplicial fans coarsening the Coxeter fan, see \cite{reading_cambrian_2006}. In~\cite{reading_cambrian_2009}, N.~Reading and D.~Speyer prove that these fans are combinatorially isomorphic to the normal fan of the polytopal realization in~\cite{chapoton_polytopal_2002}. C.~Hohlweg, C.~Lange and H.~Thomas then provided a family of $c$-generalized associahedra having $c$-Cambrian fans as normal fans by removing certain hyperplanes from the permutahedron~\cite{hohlweg_permutahedra_2011}. V.~Pilaud and C.~Stump recovered $c$-generalized associahedra by giving explicit vertex and hyperplane descriptions purely in terms of the subword complex approach introduced in the present paper~\cite{pilaud_brick_b_2011}.

\subsection{Multi-associahedra of type $A$}

In type $A_n$ for $n=m-2k-1$, the multi-cluster complex $\Dk{A_n}$ is isomorphic to the simplicial complex $\Dmk$ of $k$-triangulations of a convex $m$-gon. This simplicial complex is conjectured to be realizable as the boundary complex of a polytope\footnote{As far as we know, the first reference to this conjecture appears in \cite[Section 1]{jonsson_generalized_2005}.}. It was studied in many different contexts, see \cite[Section 1]{pilaud_multitriangulations_2009} 
for a detailed description of previous work on multi-triangulations. Apart from the most simple cases, very little is known about its polytopality. Nevertheless, this simplicial complex possesses very nice properties which makes this conjecture plausible. 
Indeed, the subword complex approach provides a simple description of the $1$-skeleton of a possible multi-associahedron (see Lemma~\ref{le:uniquevertex}), and gives a new and very simple proof that it is a vertex-decomposable triangulated sphere~\cite[Theorem 2.1]{stump_new_2011}, see also~\cite{Jon2003}. Below, we survey the known polytopal realizations of $\Dmk$ as boundary complexes of convex polytopes. 
The simplicial complex $\Dmk$, or equivalently the multi-cluster complex $\Dk{A_n}$ for $n=m-2k-1$, is the boundary complex of
\begin{itemize}
  \item a point, if $k=0$,
  \item an $n$-dimensional dual associahedron, if $k=1$,
  \item a $k$-dimensional simplex, if $n=1$,
  \item a $2k$-dimensional cyclic polytope on $2k+3$ vertices, if $n=2$, see \cite[Section 8]{pilaud_multitriangulations_2009},
  \item a $6$-dimensional simplicial polytope, if $n=3$ and $k=2$, see \cite{bokowski_symmetric_2009}.
\end{itemize}
The case $n=2$ is also a direct consequence of the rank 2 description in Section~\ref{sec:rank2}. 
Further unsuccessful attempts to realize $\Dmk$ come from various directions in discrete geometry.
\begin{enumerate}[(a)]
  \item A generalized construction of the polytope of pseudo-triangulations~\cite{rote_pseudo_2008} using rigidity of pseudo-triangulations~\cite[Section 4.2 and Remark 4.82]{pilaud_thesis_2010}.
  \item A generalized construction of the secondary polytope. As presented in \cite{gelfand_discriminants_2008}, the secondary polytope of a point configuration can be generalized using star polygons~\cite[Section 4.3]{pilaud_thesis_2010}.
  \item The brick polytope of a sorting network~\cite{pilaud_brick_2011}. This new approach brought up a large family of spherical subword complexes which are realizable as the boundary of a polytope. In particular it provides a new perspective on generalized associahedra~\cite{pilaud_brick_b_2011}. Unfortunately, this polytope fails to realize the multi-associahedron.
\end{enumerate}

\subsection{Multi-associahedra of type $B$}\label{subsec:multi-type-B}

We start by proving Theorem~\ref{th:typeB} which says that 
the multi-cluster complex $\Dk{B_{m-k}}$ is isomorphic to the simplicial complex of centrally symmetric $k$-triangulations of a regular convex $2m$-gon.
This simplicial complex was studied in~\cite{soll_type-b_2009,rubey_crossings_2010}. We then present what is known about its polytopality. The new approach using subword complexes provides in particular very simple proofs of Corollaries~\ref{cor:sym_multi},~\ref{cor:typeBvdecomp} and~\ref{cor:conj17}.
\begin{proof}[Proof of Theorem \ref{th:typeB}]
  Let $S=\{s_0,s_1,\dots, s_{m-k-1}\}$ be the generators of $B_{m-k}$, where $s_0$ is the generator such that 
  $(s_0s_1)^4=\Id \in W$, and the other generators satisfy the same relations as in type~$A_{m-k-1}$. Then, 
  embed the group $B_{m-k}$ in the group $A_{2(m-k)-1}$ by the standard folding technique: replace $s_0$ by $s'_{m-k}$ 
  and $s_i$ by $s'_{m-k+i} s'_{m-k-i}$ for $1 \leq i \leq m-k-1$, where the set $S'=\{s_1',\dots, s_{2(m-k)-1}'\}$ generates the group $A_{2(m-k)-1}$. The 
  multi-cluster complex $\Dk {B_{m-k}}$ now has an embedding into the multi-cluster complex $\Delta^{k}_{c'}(A_{2(m-k)-1})$, 
  where $c'$ is the Coxeter element of type $A_{2(m-k)-1}$ corresponding to $c$ in $B_{m-k}$; 
  the corresponding subcomplex has the property that $2(m-k)$ generators (all of them except $s'_{m-k}$) always 
  come in pairs. Using the correspondence between $k$-triangulations and facets of the multi-cluster complex described in Section~\ref{sec:results}, the facets of $\Dk {B_{m-k}}$ considered in $\Delta^{k}_{c'}(A_{2(m-k)-1})$ correspond to centrally symmetric multi-triangulations.
\end{proof}
Here, we present the few cases for which this simplicial complex is known to be polytopal. The multi-cluster complex $\Dk {B_{m-k}}$ is the boundary complex of
\begin{itemize}
  \item an $(m-1)$-dimensional dual cyclohedron (or type $B$ associahedron), if $k=1$, see \cite{simion_type-b_2003, hohlweg_realizations_2007},
  \item an $(m-1)$-dimensional simplex, if $k=m-1$,
  \item a $(2m-4)$-dimensional cyclic polytope on $2m$ vertices, if $k=m-2$, see \cite{soll_type-b_2009}.
\end{itemize}
The case $k=m-2$ also follows from the rank 2 description in Section \ref{sec:rank2}.

\subsection{Generalized multi-associahedra of rank $2$}\label{sec:rank2}

We now prove that multi-cluster complexes of rank $2$ can be realized as boundary complexes of cyclic polytopes. In other words, we show the existence of rank $2$ multi-associahedra. This particular case was known independently by~D.~Armstrong\footnote{Personal communication.}.
\begin{theorem}[Type $I_2(m)$ multi-associahedra] \label{ex:I_2}
  The multi-cluster complex $\Dk {I_2(m)}$ is isomorphic to the boundary complex of a $2k$-dimensional cyclic polytope on $2k+m$ vertices. The multi-associahedron of type $I_2(m)$ is the simple polytope given by the dual of a 
  $2k$-dimensional cyclic polytope on $2k+m$ vertices.
\end{theorem}
\begin{proof}
   This is obtained by Gale's evenness criterion on the word $Q=(a,b,a,b,a,\dots)$ of length $2k+m$: Let $F$ be a facet of $\Dk {I_2(m)}$, and take two consecutive letters $x$ and $y$ in the complement of $F$. Since the complement of $F$ is a reduced expression of $\wo$, then $x$ and $y$ must represent different generators. Since the letters in $Q$ are alternating, it implies that the number of letters between $x$ and $y$ is even.
\end{proof}

\subsection{Generalized multi-associahedra}\label{subsec:Genmultiasso}

Recall from Section~\ref{sec:subword complexes} that a subword complex $\Delta(Q,\pi)$ is homeomorphic to a sphere if and 
only if the Demazure product $\delta(Q)=\pi$, and to a ball otherwise. This 
motivates the question whether spherical subword complexes can be realized as boundary complexes of polytopes 
\cite[Question 6.4.]{knutson_subword_2004}. We show that it is enough to consider multi-cluster complexes to prove 
polytopality for all spherical subword complexes, and we characterize simplicial spheres that can be realized 
as subword complexes in terms of faces of multi-cluster complexes.

\begin{lemma} 
  Every spherical subword complex $\Delta(Q,\wo)$ is the link of a face of a multi-cluster complex 
  $\Delta(\ckwoc,\wo)$.
\end{lemma}

\begin{proof}
  Observe that any word $Q$ in $S$ can be embedded as a subword of $Q'=\ckwoc$, for $k$ less than or equal to the 
  size of $Q$, by assigning the $i$-th letter of $Q$ within the $i$-th copy of $\c$. Since the Demazure product $\delta (Q)$ 
  is equal to $\wo$, the word $Q$ contains a reduced expression for $\wo$. In other words, the set $Q'\setminus Q$ is a 
  face of $\Delta(Q',\wo)$. The link of this face in $\Delta(Q',\wo)$ consists of subwords of $Q$ -- viewed as a subword of 
  $Q'$ -- whose complements contain a reduced expression of $\wo$. This corresponds exactly to the subword complex~$\Delta(Q,\wo)$.
\end{proof}

We now prove that simplicial spheres realizable as subword complexes are links of faces of multi-cluster complexes.

\begin{proof}[Proof of Theorem~\ref{thm:universal}]
  For any spherical subword complex $\Delta(Q,\pi)$, we have that the Demazure product $\delta(Q)$ equals $\pi$. By Theorem~\ref{th:subword complexes for w0}, $\Delta(Q,\pi)$ is isomorphic to a subword complex of the form $\Delta(Q',\wo)$. Using the previous lemma we obtain that $\Delta(Q,\pi)$ is the link of a face of a multi-cluster complex. 
The other direction follows since the link of a subword (i.e., a face) of a multi-cluster complex is itself a subword complex, corresponding to the complement of this subword.
\end{proof}

Finally we prove that the question of polytopality of spherical subword complexes is equivalent to the question of polytopality of multi-cluster complexes.

\begin{proof}[Proof of Corollary \ref{cor:polytopality}]
  On one hand, if every spherical subword complex is polytopal then clearly every multi-cluster complex is polytopal. On the other 
  hand, suppose that every spherical subword complex is polytopal. Every spherical subword complex is the link of a face of a 
  multi-cluster complex. Since the link of a face of a polytope is also polytopal, Theorem~\ref{thm:universal} implies that every spherical subword complex is polytopal.
\end{proof}

%
%

\section{Sorting words of the longest element and the SIN-property} \label{sec:sufficient}

In this section, we give a simple combinatorial description of the $c$-sorting words of $\wo$, and 
prove that a word $Q$ coincides up to commutations with $\ckwoc$ for some non-negative integer $k$ if and only if $Q$ has the SIN-property as defined in Section~\ref{sec:results}. This gives us an alternative way of defining multi-cluster complexes in terms of words having the SIN-property. Recall the involution $\psi : S \rightarrow S$ from Section~\ref{sec:sorting} defined by $\psi(s) = \wo^{-1} s \wo$. The sorting word of $\wo$ has the following important property.

\begin{proposition}\label{prop:suffix}
  The sorting word $\woc$ is, up to commutations, equal to a word with suffix 
  $(\psi(c_1),\dots, \psi(c_n))$, where $c=c_1\cdots c_n$.
\end{proposition}

\begin{proof}
  As $\wo$ has a $c$-sorting word having $\c = (c_1,\ldots,c_n)$ as a prefix, the corollary is obtained by applying 
  Proposition~\ref{prop:init_letter} $n$ times.
\end{proof}

Given a word ${\bf w}$ in $S$, define the function $\phi_{\bf w}: S \rightarrow \mathbb{N}$ given by $\phi_{\bf w} (s)$ being the number of occurrences of the letter $s$ in ${\bf w}$.

\begin{theorem} \label{thm:charac_phi}
  Let $\woc$ be the $c$-sorting word of $\wo$ and let $s,t$ be neighbors in the Coxeter graph such that $s$ comes before $t$ in $c$. Then
  $$
  \phi_{\woc}(s) =
  \begin{cases}
    \phi_{\woc}(t) & \text{if $\psi(s)$ comes before $\psi(t)$ in $c$}, \\
    \phi_{\woc}(t)+1 & \text{if $\psi(s)$ comes after $\psi(t)$ in $c$}.
  \end{cases}
  $$
\end{theorem}

\begin{proof}
  Sorting words of $\wo$ have intervening neighbors, see~\cite[Proposition~2.1]{speyer_powers_2009} for an equivalent formulation. Therefore $s$ and $t$ alternate in $\woc$, with $s$ coming first. Thus, $\phi_{\woc}(s) = \phi_{\woc}(t)$ if and only if the last $t$ comes after the last $s$. Using Proposition \ref{prop:suffix}, this means that $s$ appears before $t$ in $\psi(\c)$ or equivalently $\psi(s)$ appear before $\psi(t)$ in $c$. Otherwise, the last $s$ will appear after the last $t$.
\end{proof}

It is known that if $\psi$ is the identity on $S$, or equivalently if $\wo = -\Id$, then the $c$-sorting word of $\wo$ is given by 
$\woc=\c^{\frac{h}{2}}$, where $h$ denotes the \Dfn{Coxeter number} given by the order of any Coxeter element. In the case where $\psi$ is not the identity on $S$ (that is when $W$ is of types $A_n$ ($n\geq 2$), $D_n$ ($n$ odd), $E_6$ and $I_2(m)$ ($m$ odd), see~\cite[Exercise 10 of Chapter 4]{bjoerner_combinatorics_2005}), the previous theorem gives simple way to obtain the sorting words of $\wo$.
\begin{algorithm}\label{algo:auslanderreiten}
Let $W$ be an irreducible finite Coxeter group, and let $c = c_1c_2 \cdots c_n$ be a Coxeter element.
  \begin{enumerate}[(i)]
    \item Since the Coxeter diagram is connected, one can use Theorem~\ref{thm:charac_phi} to compute $\phi_{\woc}(s)$ for all $s$ depending on $m := \phi_{\woc}(c_1)$;
    \item using that the number of positive roots equals $nh/2$, one obtains $m$ and thus all $\phi_{\woc}(s)$ using $$2 \cdot \sum_{s \in S} \phi_{\woc}(s) = nh.$$
    \item using that $\woc = \c_{K_1} \c_{K_2} \cdots \c_{K_r}$ where $K_i\subseteq S$ for $1\leq i\leq r$ and $c_I$, with $I\subseteq S$, 
  is the Coxeter element of $W_I$ obtained from $c$ by keeping only letters in $I$, we obtain that $\c_{K_i}$ is the product of all $s$ for which $\phi_{\woc}(s) \geq i$.
  \end{enumerate}
\end{algorithm}
This algorithm provides an explicit description of the sorting words of the longest element $\wo$ of any finite Coxeter group using nothing else than Coxeter group theory. This answers a question raised in~\cite[Remark~2.3]{hohlweg_permutahedra_2011} and simplifies a step in the construction of the $c$-generalized associahedron. We now give two examples of how to use this algorithm.

\begin{example}\label{ex:A4}
  Let $W=A_4$ and $S=\{s_1,s_2,s_3, s_4\}$ with the labeling of the graph shown in Figure~\ref{fig:AR} on page~\pageref{fig:AR}. Moreover, let~$c=s_1s_3s_2s_4$. Fix $\phi_{\woc}(s_1)=m$. Since $s_1$ comes 
  before $s_2$ in $c$ and that $\psi(s_1)=s_4$ comes after $\psi(s_2)=s_3$, the letter $s_1$ appears one more time than 
  the letter $s_2$ in $\woc$, i.e., $\phi_{\woc}(s_2)=m-1$. Repeating the same argument gives 
  $\phi_{\woc}(s_3) = m \text{ and } \ \phi_{\woc}(s_4) = m-1$. Summing up these values gives the equality 
  $4m-2=\frac{n\cdot h}{2}=\frac{4\cdot 5}{2}=10$, and thus $m=3$. Finally, the $c$-sorting word is 
  $\woc=(s_1,s_3,s_2,s_4 | s_1,s_3,s_2,s_4 | s_1,s_3)$.
\end{example}
\begin{example}\label{ex:E6}
  Let $W=E_6$ and $S=\{s_1,s_2,\dots, s_6\}$ with the labeling of the graph shown in Figure~\ref{fig:AR} on page~\pageref{fig:AR}. Moreover, let 
  $c=s_3s_5s_4s_6s_2s_1$. Fix $\phi_{\woc}(s_6)=m$. Repeating the same procedure from 
  the previous example and using that $\psi(s_6)=s_6$, $\psi(s_3)=s_3$, $\psi(s_2)=s_5$, $\psi(s_1)=s_4$, one get 
  $\phi_{\woc}(s_1) = \phi_{\woc}(s_2) =  {m-1}, \ \phi_{\woc}(s_3) = \phi_{\woc}(s_6) = m, \ 
  \phi_{\woc}(s_4) = \phi_{\woc}(s_5) = m+1$. As the sum equals $\frac{nh}{2} = \frac{6\cdot 12}{2} = 36$, we obtain $m=6$. Finally, the $c$-sorting word is $(\c^5 | s_3,s_5,s_4,s_6 | s_5,s_4)$.
\end{example}

\begin{remark} \label{rem:computation}
  Propositions~\ref{prop:init_letter} and \ref{prop:suffix} have the following computational consequences. Denote by 
  $\operatorname{rev}({\bf w})$ the reverse of a word ${\bf w}$. First, up to commutations, we have
    $$\woc = \operatorname{rev}(\woword(\psi(\operatorname{rev}(\c)))).$$
  Second, we also have, up to commutation,
    $$\c^h = \woc \operatorname{rev}(\woword(\operatorname{rev}(\c))).$$
  Third, for all $s\in S$,
  $$\phi_{\woc}(s)+\phi_{\mathbf{w}_\circ(\operatorname{rev}(\c))}(s) = \phi_{\woc}(s)+\phi_{\woc}(\psi(s)) = h.$$
\end{remark}
We are now in the position to prove Theorem~\ref{thm:if_part_conj}.
\begin{proof}[Proof of Theorem~\ref{thm:if_part_conj}]
  Suppose that a word $Q$ has the SIN-property, then it has complete support by definition, and it contains, up to commutations, some word $\c=(c_1,\dots, c_n)$ for a Coxeter element $c$ as a prefix. Moreover, the word $(\psi(c_1),\dots, \psi(c_n))$ is a suffix of $Q$, up to commutations. 
  Observe that a word has intervening neighbors if and only if it is a prefix of $\c^\infty$ up to commutations, see \cite[Section~3]{eriksson_conjugacy_2009}. In view of Lemma~\ref{le:speyer} and the equality $\delta(Q)=\wo$, the word $Q$ has, up to commutations, $\woc$ as a prefix. If the length of $Q$ equals $\wo$ 
  the proof ends here with $k=0$. Otherwise, the analogous argument for $\operatorname{rev}(Q)$ gives that the word $\operatorname{rev}(Q)$ has, up to commutations, $\woword(\psi(\operatorname{rev}(\c)))$ as a prefix. By Remark~\ref{rem:computation}, the word~$\woword(\psi(\operatorname{rev}(\c)))$ is, up commutations, equal to the reverse of $\woc$. Therefore,~$Q$ has the word $\woc$ also as a suffix. Since $\c=(c_1,\dots, c_n)$ is a prefix of $Q$ and of $\woc$, and $Q$ has intervening neighbors, $Q$ coincides with $\ckwoc$ up to commutations.
  Moreover, if $Q$ is equal to $\ckwoc$ up to commutations, it has intervening neighbors, and a suffix $(\psi(c_1),\dots, \psi(c_n))$, up to commutations, by Proposition \ref{prop:suffix}. This implies that the word $Q$ has the SIN-property.
\end{proof}

\begin{remark}\label{re:choosing}
  In light of Theorem~\ref{thm:if_part_conj} and Section~\ref{subsec:isomorphic subword complexes}, starting with a word $Q$ having the SIN-property suffices to construct a multi-cluster complex, and choosing a particular connected subword in the bi-infinite word $\widetilde Q$, defined in Section~\ref{subsec:isomorphic subword complexes}, corresponds to choosing a particular Coxeter element.
\end{remark}

We finish this section with a simple observation on the bi-infinite word $\widetilde Q$. For any letter $q$ in the word $Q\psi(Q)$, let $\beta_q$ be the root obtained by applying the prefix $w_q$ of $Q\psi(Q)$ before $q$ to the simple root $\alpha_q$. To obtain roots for all letters in $\widetilde Q$, repeat this association periodically.

\begin{proposition}\label{pr:last reflection on SIN words}
  Let $Q$ be a word in $S$ having the SIN-property, and let $q,q'$ be two consecutive occurrences of the same letter $s$ in $\widetilde Q$. Then
  $$\beta_q+\beta_{q'} = \sum_{p} -a_{sp} \beta_{p},$$
  where the sum ranges over the collection of letters $p$ in $\widetilde Q$ between $q$ and $q'$ corresponding to neighbors of $s$ in the Coxeter graph, and where $(a_{st})_{s,t \in S}$ is the corresponding \Dfn{Cartan matrix}.
\end{proposition}

\begin{proof}
  Without loss of generality, we can assume that $q$ is the first letter in some occurrence of $Q$, as otherwise, we can shift $Q$ accordingly. Let $w_{\langle s \rangle}$ be the product of all neighbors of $s$ in the Coxeter graph (in any order, as they all commute). The result follows from a direct calculation.
  $$\beta_q+\beta_{q'} = \alpha_s + sw_{\langle s \rangle}(\alpha_s) = \alpha_s + s\big( \alpha_s + \sum_{p} -a_{sp}\alpha_p \big) = \sum_{p} -a_{sp}s(\alpha_p) = \sum_{p} -a_{sp} \beta_p,$$
  where the first equality comes from the fact that $Q$ has the SIN-property, the second comes from the fact that $p(\alpha_s) = \alpha_s - a_{sp} \alpha_p$, and that any two neighbors of $s$ in the Coxeter graph commute, while the last two are trivial calculations.
\end{proof}

%
%

\section{Multi-cluster complexes, Auslander-Reiten quivers, and repetition quivers} \label{sec:cyclic action}

In this section, we connect multi-cluster complexes to Auslander-Reiten quivers and repetition quivers.
This approach emphasizes that the multi-cluster complex can be seen as a cyclic object which does not depend on a particular choice of a Coxeter element.
In type~$A$, this approach corresponds to considering subword complexes on a M\"obius strip, see~\cite{pilaud_multitriangulations_2010}.
We use this connection to introduce an action on vertices and facets of multi-cluster complexes generalizing the natural rotation action on multi-triangulations in type~$A$.

Auslander-Reiten and repetition quivers play a crucial role in \emph{Auslander-Reiten theory} which studies the representation theory of Artinian rings and quivers. The Auslander-Reiten quiver $\Gamma_\Omega$ of a quiver $\Omega$ encodes the irreducible morphisms between isomorphism classes of indecomposable representations of right modules over $\mathbf{k}\Omega$. These were introduced by M.~Auslander and I.~Reiten in~\cite{auslander_almost_1974,AR1975}. We also refer to~\cite{ARS1995,gabriel_auslander-reiten_1980} for further background.
We use these connections to describe a natural cyclic action on multi-cluster complexes generalizing the rotation of the polygon in types $A$ and $B$, see Theorems~\ref{thm:order} and~\ref{thm:cyclic-action}.

\subsection{The Auslander-Reiten quiver}

In types $A$, $D$ and $E$, sorting words of $\wo$ are intimately related to Auslander-Reiten quivers. Starting with a quiver $\Omega_c$ associated to a Coxeter element $c$ (as described in Section~\ref{sec:mainresults}), one can construct combinatorially the Auslander-Reiten quiver $\Gamma_{\Omega_c}$, see~\cite[Section 2.6]{bedard_commutation_1999}. R.~B\'edard then shows how to obtain all reduced expressions for $\wo$ adapted to $\Omega_c$ (i.e., the words equal to $\woc$ up to commutations) using the Auslander-Reiten quiver and a certain tableau. K.~Igusa and R.~Schiffler use these connections in order to obtain their description of $\cwoc$, see~\cite[Sections 2.1--2.3]{igusa_exceptional_2010}. Conversely, given the $c$-sorting word $\woc$, one can recover the Auslander-Reiten quiver $\Gamma_{\Omega_c}$, see~\cite[Proposition 1.2]{zelikson_auslander-reiten_2005} and the discussion preceding it. Algorithm~\ref{algo:auslanderreiten} thus provides a way to construct the Auslander-Reiten quiver in finite types using only Coxeter group theory; it uses results on admissible sequences~\cite{speyer_powers_2009} and words with intervening neighbors~\cite{eriksson_words_2010}.

\begin{algorithm}\label{algo:auslanderreiten_extention}
  The following fourth step added to Algorithm~\ref{algo:auslanderreiten} yields the Auslander-Reiten quiver $\Gamma_{\Omega_c}$ of $\Omega_c$.
  \begin{enumerate}
    \item[(iv)] The vertices of $\Gamma_{\Omega_c}$ are the letters of $\woc$ and two letters $q,q'$ of $\woc$ are linked by an arrow $q\longrightarrow q'$ in $\Gamma_{\Omega_c}$ if and only if $q$ and $q'$ are neighbors in the Coxeter graph and $q$ comes directly before $q'$ in $\woc$ when restricted to the letters $q$ and $q'$.
  \end{enumerate}
\end{algorithm}

Figure~\ref{fig:AR} shows two examples of Auslander-Reiten quivers and how to obtain it using this algorithm.

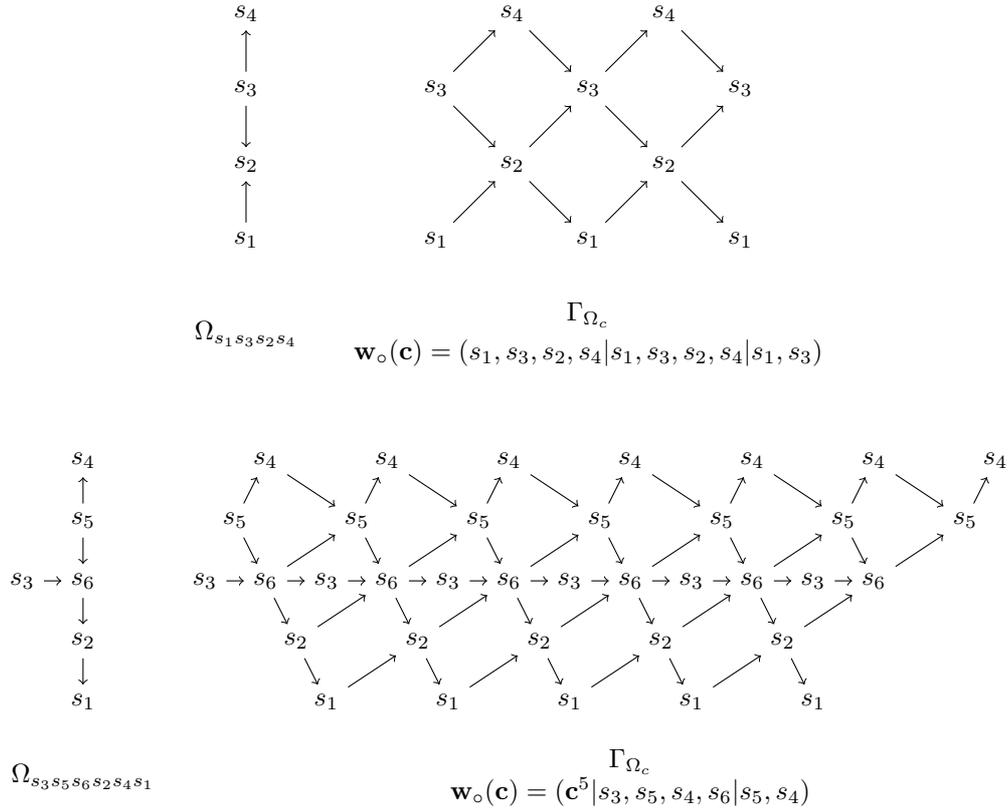
\begin{figure}[!t]
  \begin{center}
  \begin{tikzpicture}[spec/.style={gray}]
  
  \node (s3-1) at (0,2) {$s_3$};
  \node (s4-1) at (1,3) {$s_4$} edge[<-] (s3-1);
  \node (s1-1) at (0,0) {$s_1$};
  \node (s2-1) at (1,1) {$s_2$} edge[<-] (s3-1)  edge[<-] (s1-1);
  
  \node (s3-2) at (2,2) {$s_3$} edge[<-] (s4-1)  edge[<-] (s2-1);
  \node (s4-2) at (3,3) {$s_4$} edge[<-] (s3-2);
  \node (s1-2) at (2,0) {$s_1$} edge[<-] (s2-1);
  \node (s2-2) at (3,1) {$s_2$} edge[<-] (s3-2)  edge[<-] (s1-2);
  
  \node (s3-3) at (4,2) {$s_3$} edge[<-] (s4-2)  edge[<-] (s2-2);
  \node (s1-3) at (4,0) {$s_1$} edge[<-] (s2-2);
  
  \node (s1-0) at (-2.5,0) {$s_1$};
  \node (s2-0) at (-2.5,1) {$s_2$} edge[<-] (s1-0);
  \node (s3-0) at (-2.5,2) {$s_3$} edge[->] (s2-0);
  \node (s4-0) at (-2.5,3) {$s_4$} edge[<-] (s3-0);
  
  \node (c) at (-2.5,-1.25) {$\Omega_{s_1s_3s_2s_4}$};
  \node (wo) at (2,-1) {$\Gamma_{\Omega_c}$};
  \node at (2,-1.5) {$\woc=(s_1,s_3,s_2,s_4 | s_1,s_3,s_2,s_4 | s_1,s_3)$};
  \end{tikzpicture}

\vspace*{25pt}

  \begin{tikzpicture}[scale=0.8,spec/.style={gray}]
  
  \node (s3-1) at (0,2) {$s_3$};
  \node (s5-1) at (0.5,3) {$s_5$};
  \node (s6-1) at (1,2) {$s_6$} edge[<-] (s3-1)  edge[<-] (s5-1);
  \node (s4-1) at (1,4) {$s_4$} edge[<-] (s5-1);
  \node (s2-1) at (1.5,1) {$s_2$} edge[<-] (s6-1);
  \node (s1-1) at (2,0) {$s_1$} edge[<-] (s2-1);

  \node (s3-2) at (2,2) {$s_3$};
  \node (s5-2) at (2.5,3) {$s_5$};
  \node (s6-2) at (3,2) {$s_6$} edge[<-] (s3-2)  edge[<-] (s5-2);
  \node (s4-2) at (3,4) {$s_4$} edge[<-] (s5-2);
  \node (s2-2) at (3.5,1) {$s_2$} edge[<-] (s6-2);
  \node (s1-2) at (4,0) {$s_1$} edge[<-] (s2-2);
  
  \node (s3-3) at (4,2) {$s_3$};
  \node (s5-3) at (4.5,3) {$s_5$};
  \node (s6-3) at (5,2) {$s_6$} edge[<-] (s3-3)  edge[<-] (s5-3);
  \node (s4-3) at (5,4) {$s_4$} edge[<-] (s5-3);
  \node (s2-3) at (5.5,1) {$s_2$} edge[<-] (s6-3);
  \node (s1-3) at (6,0) {$s_1$} edge[<-] (s2-3);

  \node (s3-4) at (6,2) {$s_3$};
  \node (s5-4) at (6.5,3) {$s_5$};
  \node (s6-4) at (7,2) {$s_6$} edge[<-] (s3-4)  edge[<-] (s5-4);
  \node (s4-4) at (7,4) {$s_4$} edge[<-] (s5-4);
  \node (s2-4) at (7.5,1) {$s_2$} edge[<-] (s6-4);
  \node (s1-4) at (8,0) {$s_1$} edge[<-] (s2-4);
  
  \node (s3-5) at (8,2) {$s_3$};
  \node (s5-5) at (8.5,3) {$s_5$};
  \node (s6-5) at (9,2) {$s_6$} edge[<-] (s3-5)  edge[<-] (s5-5);
  \node (s4-5) at (9,4) {$s_4$} edge[<-] (s5-5);
  \node (s2-5) at (9.5,1) {$s_2$} edge[<-] (s6-5);
  \node (s1-5) at (10,0) {$s_1$} edge[<-] (s2-5);
  
  \node (s3-6) at (10,2) {$s_3$};these 
  \node (s5-6) at (10.5,3) {$s_5$};
  \node (s6-6) at (11,2) {$s_6$} edge[<-] (s3-6)  edge[<-] (s5-6);
  \node (s4-6) at (11,4) {$s_4$} edge[<-] (s5-6);
  
  \node (s5-7) at (12.5,3) {$s_5$};
  \node (s4-7) at (13,4) {$s_4$} edge[<-] (s5-7);
  
  \draw[->] (s4-1) -- (s5-2);
  \draw[->] (s6-1) -- (s5-2);
  \draw[->] (s6-1) -- (s3-2);
  \draw[->] (s2-1) -- (s6-2);
  \draw[->] (s1-1) -- (s2-2);
  
  \draw[->] (s4-2) -- (s5-3);
  \draw[->] (s6-2) -- (s5-3);
  \draw[->] (s6-2) -- (s3-3);
  \draw[->] (s2-2) -- (s6-3);
  \draw[->] (s1-2) -- (s2-3);
  
  \draw[->] (s4-3) -- (s5-4);
  \draw[->] (s6-3) -- (s5-4);
  \draw[->] (s6-3) -- (s3-4);
  \draw[->] (s2-3) -- (s6-4);
  \draw[->] (s1-3) -- (s2-4);
  
  \draw[->] (s4-4) -- (s5-5);
  \draw[->] (s6-4) -- (s5-5);
  \draw[->] (s6-4) -- (s3-5);
  \draw[->] (s2-4) -- (s6-5);
  \draw[->] (s1-4) -- (s2-5);
  
  \draw[->] (s4-5) -- (s5-6);
  \draw[->] (s6-5) -- (s5-6);
  \draw[->] (s6-5) -- (s3-6);
  \draw[->] (s2-5) -- (s6-6);

  \draw[->] (s4-6) -- (s5-7);
  \draw[->] (s6-6) -- (s5-7);

  \node (s1-0) at (-2,0) {$s_1$};
  \node (s2-0) at (-2,1) {$s_2$} edge[->] (s1-0);
  \node (s6-0) at (-2,2) {$s_6$} edge[->] (s2-0);
  \node (s5-0) at (-2,3) {$s_5$} edge[->] (s6-0);
  \node (s4-0) at (-2,4) {$s_4$} edge[<-] (s5-0);
  \node (s3-0) at (-3,2) {$s_3$} edge[->] (s6-0);

  \node (c) at (-2,-1.25) {$\Omega_{s_3s_5s_6s_2s_4s_1}$};
  \node (wo) at (7,-1) {$\Gamma_{\Omega_c}$};
  \node at (7,-1.5) {$\woc=(\c^5 | s_3,s_5,s_4,s_6 | s_5,s_4)$};
    
  \end{tikzpicture}
  \caption{\label{fig:AR} Two examples of Auslander-Reiten quivers of types $A_4$ and $E_6$.}
  \end{center}
\end{figure}

\subsection{The repetition quiver} Next, we define the repetition quiver.

\begin{definition}[{\cite[Section 2.2]{keller_cluster_2010}}]
  The \Dfn{repetition quiver} $\Z \Omega$ of a quiver $\Omega$ consists of vertices $(i,v)$ for a vertex $v$ of $\Omega$ and $i \in \Z$. The arrows of $\Z \Omega$ are given by $(i,v) \longrightarrow (i,v')$ and $(i,v') \longrightarrow (i+1,v)$, for any arrow $v \longrightarrow v'$ in $\Omega$.
\end{definition}

For a Coxeter element $c$, the repetition quiver $\Z \Omega_c$ turns out to be a bi-infinite sequence of Auslander-Reiten quivers $\Gamma_{\Omega_c}$ and $\Gamma_{\Omega_{\psi(\c)}}$ linked at the initial $\c$ and the final $\psi(\c)$. More precisely, the repetition quiver $\Z \Omega_c$ can be obtained applying the procedure described in Algorithm~\ref{algo:auslanderreiten_extention} to the bi-infinite word
$$\widetilde \woc = \cdots \woc \hspace{5pt} \psi(\woc) \hspace{5pt} \woc \hspace{5pt} \psi(\woc) \cdots.$$
As discussed in Remark~\ref{re:choosing}, the word $\widetilde \woc$ does not depend on the choice of a Coxeter element. Therefore, the repetition quiver is independent of the choice of Coxeter element, as expected. The repetition quiver comes equipped with the \Dfn{Auslander-Reiten translate} $\tau$ given by $\tau(i,v) = (i-1,v)$. A second natural map acts on the vertices of the repetition quiver: the \Dfn{shift operation} $[1]: \Z \Omega_c \longrightarrow \Z \Omega_c$ which sends a vertex in $\woc$ or $\psi(\woc)$ to the corresponding vertex in the next (to the right) copy of~$\psi(\woc)$ or of~$\woc$ respectively. In Figure~\ref{fig:repetition}, we present an example of a repetition quiver of type $A_4$. Copies of the Auslander-Reiten quivers $\Gamma_{\Omega_c}$ and $\Gamma_{\Omega_{\psi(\c)}}$ are separated by dashed arrows. The Auslander-Reiten translate $\tau$ sends a vertex to the one located directly to its left. One orbit of the shift operation shown in {\bf bold}, we have $(4,s_1) = [1](1,s_4) = [2](-1,s_1)$.

\begin{figure}[!htbp]
  \begin{center}
  \begin{tikzpicture}[scale=0.85]
  \node(left1) at (-1.25,1.5){};
  \node(left2) at (-0.5,1.5){} edge[line width=1pt,dotted,-] (left1);

  \node (s2-0) at (-1,1) {};
  \node (s4-0) at (-1,3) {};

  \node (s1-1) at (0,0) {\small$(-2,s_1)$} edge[dotted,gray,<-,thick] (s2-0);
  \node (s3-1) at (0,2) {\small$(-2,s_3)$} edge[dotted,gray,<-,thick] (s4-0) edge[dotted,gray,<-,thick] (s2-0);
  \node (s2-1) at (1,1) {\small$(-2,s_2)$} edge[<-,thick] (s3-1)  edge[<-,thick] (s1-1);
  \node (s4-1) at (1,3) {\small$(-2,s_4)$} edge[<-,thick] (s3-1);
  
  \node (s1-2) at (2,0) {\small${\bf (-1,s_1)}$} edge[<-,thick] (s2-1);
  \node (s3-2) at (2,2) {\small$(-1,s_3)$} edge[<-,thick] (s4-1)  edge[<-,thick] (s2-1);
  \node (s2-2) at (3,1) {\small$(-1,s_2)$} edge[<-,thick] (s3-2)  edge[<-,thick] (s1-2);
  \node (s4-2) at (3,3) {\small$(-1,s_4)$} edge[<-,thick] (s3-2);
  
  \node (s1-3) at (4,0) {\small$(0,s_1)$} edge[<-,thick] (s2-2);
  \node (s3-3) at (4,2) {\small$(0,s_3)$} edge[<-,thick] (s4-2)  edge[<-,thick] (s2-2);


  \node (s2-3) at (5,1) {\small$(0,s_2)$} edge[dashed,<-] (s3-3)  edge[dashed,<-] (s1-3);
  \node (s4-3) at (5,3) {\small$(0,s_4)$} edge[dashed,<-] (s3-3);
  \node (s1-4) at (6,0) {\small$(1,s_1)$} edge[<-,thick] (s2-3);
  \node (s3-4) at (6,2) {\small$(1,s_3)$} edge[<-,thick] (s2-3)  edge[<-,thick] (s4-3);
  
  \node (s2-4) at (7,1) {\small$(1,s_2)$} edge[<-,thick] (s3-4)  edge[<-,thick] (s1-4);
  \node (s4-4) at (7,3) {\small${\bf (1,s_4)}$} edge[<-,thick] (s3-4);
  \node (s1-5) at (8,0) {\small$(2,s_1)$} edge[<-,thick] (s2-4);
  \node (s3-5) at (8,2) {\small$(2,s_3)$} edge[<-,thick] (s2-4)  edge[<-,thick] (s4-4);
  
  \node (s2-5) at (9,1) {\small$(2,s_2)$} edge[<-,thick] (s3-5)  edge[<-,thick] (s1-5);
  \node (s4-5) at (9,3) {\small$(2,s_4)$} edge[<-,thick] (s3-5);


  \node (s1-6) at (10,0) {\small$(3,s_1)$} edge[dashed,<-] (s2-5);
  \node (s3-6) at (10,2) {\small$(3,s_3)$} edge[dashed,<-] (s4-5)  edge[dashed,<-] (s2-5);
  \node (s2-6) at (11,1) {\small$(3,s_2)$} edge[<-,thick] (s1-6)  edge[<-,thick] (s3-6);
  \node (s4-6) at (11,3) {\small$(3,s_4)$} edge[<-,thick] (s3-6);
  
  \node (s1-7) at (12,0) {\small${\bf (4,s_1)}$} edge[<-,thick] (s2-6);
  \node (s3-7) at (12,2) {\small$(4,s_3)$} edge[<-,thick] (s2-6)  edge[<-,thick] (s4-6);
  \node (s2-7) at (13,1) {\small$(4,s_2)$} edge[<-,thick] (s1-7)  edge[<-,thick] (s3-7);
  \node (s4-7) at (13,3) {\small$(4,s_4)$} edge[<-,thick] (s3-7);
  
  \node (s1-8) at (14,0) {\small$(5,s_1)$} edge[<-,thick] (s2-7);
  \node (s3-8) at (14,2) {\small$(5,s_3)$} edge[<-,thick] (s2-7)  edge[<-,thick] (s4-7);

  \node(right1) at (15.25,1.5){};
  \node(right2) at (14.5,1.5){} edge[line width=1pt,dotted,-] (right1);

  \node (s2-9) at (15,1) {} edge[dotted,gray,<-,thick] (s1-8)  edge[dotted,gray,<-,thick] (s3-8);
  \node (s4-9) at (15,3) {} edge[dotted,gray,<-,thick] (s3-8);

  \node (a) at (2,-1) {$\underbrace{\hspace{4.1cm}}_{\Gamma_{\Omega_c}}$};
  \node (a) at (7,-1.05) {$\underbrace{\hspace{4.1cm}}_{\Gamma_{\Omega_{\psi(\c)}}}$};
  \node (a) at (12,-1) {$\underbrace{\hspace{4.1cm}}_{\Gamma_{\Omega_c}}$};

  \draw[dashed,red] (-1.1,-0.3)--(2.6,3.3)--(5.3,3.3)--(9,-0.3) -- cycle;
  \draw[dashed,blue] (6.5,-0.5)--  (8,1) -- (7,2) -- (8.5,3.5) -- (13.5,3.5) -- (15,2) -- (14,1) -- (15.5,-0.5) -- cycle;
  \end{tikzpicture}
  \caption{\label{fig:repetition} The repetition quiver of type $A_4$ with the quiver $\Omega_c$ associated to the Coxeter element $c = s_1s_3s_2s_4$.}
  \end{center}
\end{figure}
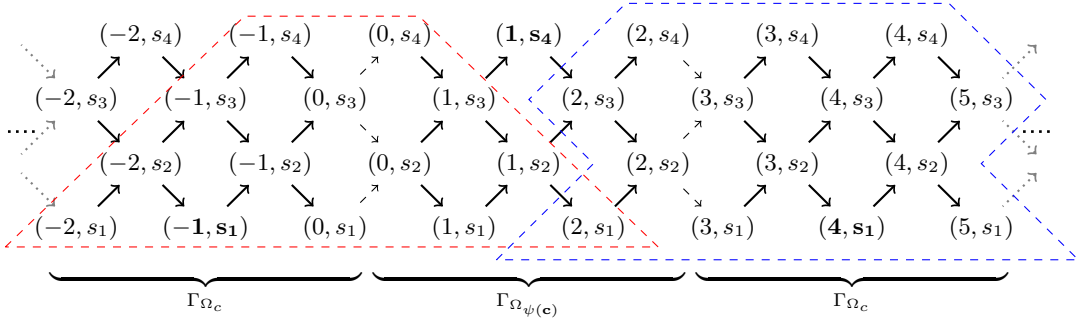

\begin{remark}
  Vertices in the Auslander-Reiten quiver correspond to (isomorphism classes of) indecomposable representations of $\Omega_c$, and thus have a \emph{dimension vector} attached. By the \emph{knitting algorithm}, the dimension vector at a vertex $V = (i,v)$ of $\Gamma_{\Omega_c}$ plus the dimension vector at the vertex $\tau(V)$ equals the sum of all dimension vectors at vertices $V'$ for which $ \tau(V) \longrightarrow V' \longrightarrow V$ are arrows in $\Gamma_{\Omega_c}$, see~\cite[Section~10.2]{gabriel_representations_1997}. This procedure is intimately related to the SIN-property, which ensures that this sum is indeed over all neighbors of $v$. Moreover, Proposition~\ref{pr:last reflection on SIN words} implies that this property holds as well for the root $\beta_q$ attached to a letter $q$ in the bi-infinite sequence~$\widetilde Q$. This yields the well known property that the dimension vector and the corresponding root coincide.
\end{remark}

The following proposition describes words for the multi-cluster complex using the repetition quiver, the Auslander-Reiten translate, and the shift operation.
\begin{proposition}\label{prop:auslanderreiten}
  Let $\Omega_c$ be a quiver corresponding to a Coxeter element $c$. Words for the multi-cluster complex are obtained from the bi-infinite word $\widetilde \woc$ by setting $\tau^k = [1]$. Choosing a particular fundamental domain for this identification corresponds to choosing a particular Coxeter element. In other words, words for multi-cluster complexes are obtained by a choice of linear extension of a fundamental domain of the identification $\tau^k = [1]$ in the repetition quiver.
\end{proposition}

\begin{proof}
  With the identification $[1]V = \tau^kV$ in the repetition quiver, a fundamental domain will consist of $k$ copies of $\Omega_c$ and one copy of the Auslander-Reiten quiver $\Gamma_{\Omega_c}$. This fundamental domain is exactly the quiver formed from the word $\c^k \woc$ using Algorithm~\ref{algo:auslanderreiten_extention}. As linear extensions of this quiver correspond to words equal to $\c^k \woc$ up to commutations, the result follows.
\end{proof}

The {\color{red}red} and the {\color{blue}blue} boxes in Figure~\ref{fig:repetition} mark two particular choices of a fundamental domain for the multi-cluster complex of type $A_4$ with $k=1$ corresponding to the Coxeter elements $s_1s_2s_3s_4$ and $s_1s_3s_2s_4$ respectively.

\subsection{The Auslander-Reiten translate on multi-cluster complexes}\label{sec:ART}

The Auslander-Reiten translate gives a cyclic action on the vertices and facets of a multi-cluster complex. This action corresponds to natural actions on multi-triangulations in types $A$ and $B$, and is well studied in the case of cluster complexes.

\begin{definition}\label{def:cyclic action}
Let $Q = \c^k \woc$. The permutation $\Theta : Q\ \tilde\longrightarrow\ Q$ is given by sending a letter $q_i = s$ to the next occurrence of $s$ in $Q$, if possible, and to the first occurrence of $\psi(s)$ in $Q$ otherwise.
\end{definition}
Observe that in types $ADE$, the operation $\Theta$ corresponds to the inverse of the Auslander-Reiten translate, $\Theta = \tau^{-1}$ when considered within the repetition quiver.
\begin{proposition}
  The permutation $\Theta$ induces a cyclic action on the facets of $\Delta(Q,\wo)$.
\end{proposition}
\begin{proof}
  By Proposition~\ref{pr:rotated}, the subword complexes $\Delta(Q,\wo)$ and $\Delta(\rotatedword{Q}{s},\wo)$ are isomorphic for an initial letter $s$ in $Q$. Proposition~\ref{prop:init_letter} asserts that $\c^k \woc$ and the rotated word obtained from $\c^k \woc$ by rotating $n$ times are equal up to commutations. By construction, $\Theta$ is the automorphism of $\Delta(Q,\wo)$ given by inverse rotation of $\c$.
\end{proof}

\begin{example}
  As in Example~\ref{ex:A4}, consider $c = s_1s_3s_2s_4$ and $Q=\cwoc = (q_i : 1 \leq i \leq 14) = (\c^2|s_1s_3s_2s_4|s_1,s_3)$. After rotating along all letters in $\c$ from the right, we obtain the word $(s_3s_1s_4s_2|\c^2|s_1,s_3)$, so we have to reorder the initial $4$ letters using commutations to obtain again $(\c^3|s_1,s_3)$. Therefore, $\Theta$ permutes the letter of $Q$ along the permutation of the indices given by
  $$\left( \begin{array}{cccccccccccccc} 1&2&3&4&5&6&7&8&9&10&11&12&13&14\\5&6&7&8&9&10&11&12&13&14&2&1&4&3\end{array}\right).$$
  Here is an example of an orbit of $\Theta$.
  $$\begin{array}{llll} \{ q_1,q_2,q_3,q_4 \} \quad\mapsto_\Theta &  \{ q_{5},q_{6},q_{7},q_8 \} \quad\mapsto_\Theta & \{ q_{9},q_{10},q_{11},q_{12} \}  \quad\mapsto_\Theta & \{ q_{13},q_{14},q_{2},q_1\} \qquad \\ \mapsto_\Theta \quad \{q_4,q_{3},q_6,q_5\} & \mapsto_\Theta \quad \{ q_8,q_{7},q_{10},q_9 \} & \mapsto_\Theta \quad \{ q_{12},q_{11},q_{14},q_{13} \} &  \mapsto_\Theta \quad \{ q_{1},q_{2},q_{3},q_4 \}.
  \end{array}$$
\end{example}
To relate the permutation $\Theta$ to clusters, we recall the definition of bipartite Coxeter elements; consider a bipartition of the set $S=S_-\sqcup S_+$ such that any two generators in $S_\epsilon$ commute (this is possible since the graph of the Coxeter group is a tree), then form the Coxeter element $c^*=c_-c_+$, where $c_\epsilon=\prod_{s\in S_\epsilon}s$. Using the bijection $\mathsf{Lr}_{c^*}$ between letters in $\c^*\woword({\bf c^*})$ and almost positive roots, the cyclic action induced by $\Theta$ is equal to the action induced by the \Dfn{tropical Coxeter element}
$$\sigma_{c^*} := \prod_{s \in S_-} \sigma_s \prod_{s \in S_+} \sigma_s$$
on almost positive roots, see Section~\ref{sec:cluster complexes} for the definition of $\sigma_s$, and~\cite[Section~5.2]{armstrong_generalized_2009} for more details about tropical Coxeter elements. In the case of cluster complexes, S.~Fomin and N.~Reading computed the order of $\Theta$ \cite[Theorem~4.14]{fomin_root_2007}. Since the words $\cwoc$ are all connected via rotation along initial letters, the order of $\Theta$ does not depend on a specific choice of Coxeter element.
\begin{theorem}\label{thm:order}
  For $Q = \ckwoc$, the order of $\Theta$ is given by
  $$\operatorname{ord}(\Theta) = \begin{cases}
    k + h/2 & \text{ if } \wo = -{\bf 1}, \\
    2k + h & \text{ if } \wo \neq -{\bf 1}. \\
  \end{cases}$$
\end{theorem}
\begin{proof}
To obtain the order of this action, we consider the length of $Q$ divided by the length of $\c$ if $w_\circ \equiv -{\bf 1}$, and twice the length of $Q$ divided by the length of $\c$ otherwise. We have already seen in Algorithm~\ref{algo:auslanderreiten} that the length of $Q$ is given by $kn + nh/2$. As the length of $\c$ is given by $n$, the result follows.
\end{proof} 
\begin{remark}
  The action induced by the tropical Coxeter element on facets of the cluster complex was shown by S.-P.~Eu and T.-S.~Fu to exhibit a \emph{cyclic sieving phenomenon}~\cite{EF2008}. Therefore, the cyclic action induced by $\Theta$ exhibits a cyclic sieving phenomenon for facets of the cluster complex $\Delta(\cwoc,\wo)$ and \emph{any} Coxeter element $c$.
\end{remark}
Finally, for types $A$ and $B$, the cyclic action $\Theta : Q \tilde\longrightarrow Q$ corresponds to the cyclic action induced by rotation of the associated polygons.
\begin{theorem}\label{thm:cyclic-action}
Let $Q = \ckwoc$. In type $A_{m-2k-1}$, the cyclic action $\Theta$ on letters in $Q$ corresponds to the cyclic action induced by rotation on the set of $k$-relevant diagonals of a convex $m$-gon. In type $B_{m-k}$, the cyclic action $\Theta$ corresponds to the cyclic action induced by rotation on the set of $k$-relevant  centrally symmetric diagonals of a regular convex $2m$-gon.
\end{theorem}
\begin{proof}
  The simplicial complex of $k$-triangulations of a convex $m$-gon is isomorphic to the multi-cluster complex of type $A_{m-2k-1}$, so the order of $\Theta$ is given by $2k+h=2k+m-2k=m$ as expected. The simplicial complex of centrally symmetric $k$-triangulations of a regular convex $2m$-gon is isomorphic to the multi-cluster complex of type $B_{m-k}$, so the order of $\Theta$ equals $k+h/2=k+m-k=m$, as well. In type $A$, the result follows from the correspondence between letters in $Q$ and $k$-relevant diagonals in the $m$-gon as described in Section~\ref{sec:results}. In type $B$, the result follows from the correspondence between letters in $Q$ and $k$-relevant centrally symmetric diagonals in the $2m$-gon as described in Section~\ref{sec:results}.
\end{proof}

%
%

\section{Open problems} \label{sec:open problems}

We discuss open problems and present several conjectures. We start with two open problems concerning counting formulas for multi-cluster complexes.
\begin{openproblem}\label{open:multi catalan numbers}
  Find \emph{multi-Catalan numbers} counting the number of facets in the multi-cluster complex.
\end{openproblem}
Although a formula in terms of invariants of the group for the number of facets of the generalized cluster complex defined by S. Fomin 
and N. Reading is known \cite[Proposition~8.4]{fomin_generalized_2005}, a general formula in terms of invariants of the group for 
the multi-cluster complex is yet to be found. An explicit formula for type $A$ can be found in \cite[Corollary 17]{jonsson_generalized_2005}. 
In type $B$, a  formula was conjectured in \cite[Conjecture 13]{soll_type-b_2009} and proved in \cite{rubey_crossings_2010}\footnote{The 
proof appeared in Section 7 in the arxiv version, see {\tt http://arxiv.org/abs/0904.1097v2}.}. In the dihedral type $I_2(m)$, the number of facets of the 
multi-cluster complex is equal to the number of facets of a $2k$-dimensional cyclic polytope on $2k+m$ vertices. These three formulas can be reformulated in terms of invariants of the Coxeter groups of type $A$, $B$ and $I_2$ as follows,
\[
  \prod_{0 \leq j < k} \prod_{1 \leq i \leq n} \frac{ d_i + h + 2j }{ d_i + 2j },
\]
where $d_1 \leq \ldots \leq d_n$ are the \emph{degrees} of the corresponding group, and $h$ is its Coxeter number. In general, this product is not an integer. The smallest example we are aware of is type $D_6$ with $k=5$. Therefore, this product cannot count facets of the multi-cluster complex in general. The cyclic action $\Theta$ (see Definition~\ref{def:cyclic action}) on multi-cluster complexes might be useful to solve Open Problem~\ref{open:multi catalan numbers}, it gives rise to the following generalization.
\begin{openproblem}\label{openproblem:CSP}
  Find \emph{multi-Catalan polynomials} $f(q)$ such that the triple
  $$\Big(\big\{ \text{facets of }\Delta(\ckwoc,\wo)\big\},f(q),\Theta\Big)$$
  exhibits the cyclic sieving phenomenon as defined by V.~Reiner, D.~Stanton, and D.~White in~\cite{RSW2004}.
\end{openproblem}
In types $A$, $B$, and $I_2$, there is actually a natural candidate for $f(q)$, namely
\[
  \prod_{0 \leq j < k} \prod_{1 \leq i \leq n} \frac{ [d_i + h + 2j]_q }{ [d_i + 2j]_q },
\]
where $[m]_q = 1 + q + \ldots + q^{m-1}$ is a $q$-analogue of the integer $m$. In the case of multi-triangulations and centrally symmetric multi-triangulations, this triple is conjectured to exhibit the cyclic sieving phenomenon.\footnote{Personal communication with V.~Reiner.} The counting formula in types $A$, $B$ and $I_2$ can be enriched with a parameter $m$ such that it reduces for $k=1$ to the Fuss-Catalan numbers counting the number of facets in the generalized cluster complexes. The next open problem raises the question of finding a family of simplicial complexes that includes the generalized cluster complexes of S.~Fomin and N.~Reading and the multi-cluster complexes. 
\begin{openproblem}
  Construct a family of simplicial complexes which simultaneously contains generalized cluster complexes and multi-cluster complexes.
\end{openproblem}
The next open problem concerns a possible representation theoretic description of the multi-cluster complex in types $ADE$. For $k=1$, one can describe the compatibility by saying that $V \parallel_c V'$ if and only $\dim(\operatorname{Ext}^1(V,V')) = 0$, see~\cite{BMRRT2006}.
\begin{openproblem}
  Describe the multi-cluster complex within the repetition quiver using similar methods.
\end{openproblem}

The following problem extends the diameter problem of the associahedron to the family of multi-cluster complexes, see~\cite[Section~2.3.2]{pilaud_thesis_2010} for further discussions in the case of multi-triangulations.

\begin{openproblem}
  Find the diameter of the facet-adjacency graph of the multi-cluster complex~$\DWk$.
\end{openproblem}

Finally, we present several combinatorial conjectures on the multi-cluster complexes. We start with a conjecture concerning minimal non-faces.
\begin{conjecture}
  Minimal non-faces of the multi-cluster complex $\DWk$ have cardinality $k+1$.
\end{conjecture}
Since $w_\circ$ is $c$-sortable, we have $\c^k\woc = \c^k\c_{K_1}\c_{K_2}\cdots \c_{K_r}$ with $K_r \subseteq \ldots \subseteq K_2 \subseteq K_1$. This implies that the complement of any $k$ letters still contains a reduced expression for $\wo$. In other words, minimal non-faces have at least cardinality $k+1$. Moreover, using the connection to multi-triangulations and centrally symmetric triangulations, we see that the conjecture holds in types~$A$ and $B$. It also holds in the case of dihedral groups: it is not hard to see that the faces of the multi-cluster complex are given by subwords of $\ckwoc=(a,b,a,b,\dots)$ which do not contain $k+1$ pairwise non-consecutive letters (considered cyclically). The conjecture was moreover tested for all multi-cluster complexes of rank $3$ and $4$ with $k=2$.

In types $A$ and $I_2(m)$, there is a binary compatibility relation on the letters of~$\ckwoc$ such that the faces of the multi-cluster complex can be described as subsets avoiding~$k+1$ pairwise not compatible elements. We remark that this is not possible in general: in type $B_3$ with $k=2$, as in Example~\ref{ex:B3}, $\Delta_c^2(B_3)$ is isomorphic to the simplicial complex of centrally symmetric $2$-triangulations of a regular convex $10$-gon. Every pair of elements in the set $\mathcal{A}=\{ \symdiag{1,4}, \symdiag{4,7}, \symdiag{7,10}  \}$ is contained in a minimal non-face. But since $\mathcal{A}$ does not contain a $3$-crossing, it forms a face of~$\Delta_c^2(B_3)$.

Theorem~\ref{thm:if_part_conj} gives an alternative way of defining multi-cluster complexes as subword complexes $\Delta(Q,\wo)$ where the word $Q$ has the SIN-property. It seems that this definition covers indeed all subword complexes isomorphic to multi-cluster complexes.

\begin{conjecture}
  Let $Q$ be a word in $S$ with complete support and $\pi \in W$. The subword complex $\Delta(Q,\pi)$ is isomorphic to a 
  multi-cluster complex if and only if $Q$ has the SIN-property and $\pi=\delta(Q)=\wo$.
\end{conjecture}

The fact that $\pi=\delta(Q)$ is indeed necessary so that the subword complex is a sphere. It remains to show that $\pi=\wo$ and that $Q$ has the SIN-property. One reason for this conjecture is that if $Q$ does not have the SIN-property then it seems that the subword complex $\Delta(Q,\wo)$ has fewer facets than required. Indeed, we conjecture that multi-cluster complexes maximize the number of facets among all subword complexes with a word $Q$ of a given size. 

\begin{conjecture}\label{conj:maximal}
  Let $Q$ be any word in $S$ with $kn+N$ letters (where $N$ denotes the length of~$\wo$) and $\Delta(Q,\wo)$ 
  be the corresponding subword complex. The number of facets of $\Delta(Q,\wo)$ is less than or equal to the number of facets 
  of the multi-cluster complex $\DWk$. Moreover, if both numbers are equal, then the word $Q$ has the SIN-property.
\end{conjecture}

We remark that the previous two conjectures hold for the dihedral types~$I_2(m)$. 
In this case, the multi-cluster complex is isomorphic to the boundary complex of a cyclic polytope, which is a polytope that maximizes the number of facets among all polytopes in fixed dimension on a given number of vertices, see e.g.~\cite{ziegler}. Moreover, we present below a simple polytope theory argument in order to show that if a word does not satisfy the SIN-property, then the corresponding subword complex has strictly less facets than the multi-cluster complex.
First note that Corollary~\ref{cor:polytopality} and Theorem~\ref{ex:I_2} imply that 
all spherical subword complexes of type~$I_2(m)$ are polytopal. By the upper bound theorem, a polytope has as many facets as a cyclic polytope if and only if it is neighborly, see e.g. \cite{ziegler}. Therefore, it is enough to prove that if~$Q = (q_1,\ldots,q_r)$ with~$r=2k+m$ is a word in~$S = \{a,b\}$ containing two consecutive letters that are equal, then the subword complex~$\Delta(Q,\wo)$ is not neighborly. Since this is a $2k$-dimensional complex, this is equivalent to show that there is a set of~$k$ letters of~$Q$ which do not form a face. 
By applying rotation of letters and Proposition~\ref{pr:rotated}, we can assume without loss of generality that the last two letters of~$Q$ are equal.
Among the first~$2k+1$ letters of~$Q$, one of the generators~$a$ or~$b$ appears no more than~$k$ times. The set of these no more than~$k$ letters is not a face of the subword complex. The reason is that the reduced expressions in the complement of this set in~$Q$ have length at most~$m-1$, which is one less than the length of~$\wo$.

\medskip

In view of Corollary \ref{cor:polytopality}, the following conjecture restricts the study of~\cite[Question~6.4]{knutson_subword_2004}.

\begin{conjecture}\label{conj:polytopality}
  The multi-cluster complex is the boundary complex of a simplicial polytope.
\end{conjecture}

In types $A$ and $B$, this conjecture coincides with the conjecture on the existence of the corresponding multi-associahedra, see~\cite{jonsson_generalized_2005,soll_type-b_2009}, and Theorem~\ref{ex:I_2} shows that this conjecture is true for dihedral groups.

%
%

\section*{Acknowledgments}

The authors would like to thank Hugh Thomas for pointing them to the results in~\cite{igusa_exceptional_2010}, and Christophe Hohlweg, Carsten Lange, Vincent Pilaud, Hugh Thomas and G\"unter M. Ziegler for important comments and remarks on preliminary versions of this article. The first two authors are thankful to Drew Armstrong, Carsten Lange, Emerson Leon, Ezra Miller, Vincent Pilaud and Luis Serrano for useful discussions that took place during the $23$rd FPSAC Conference in Reykjavik. They are particularly grateful to Carsten Lange and Emerson Leon for numerous fruitful discussions that took place in the Arnimallee $2$ Villa in Berlin.

We used the computer algebra system {\tt Sage}~\cite{sage} for implementing the discussed objects, and to test the conjectures.

%
%

\bibliographystyle{amsalpha}
\bibliography{CLS2013}

\end{document}